\documentclass[12pt]{amsart}
\usepackage{amssymb}
\usepackage{amsthm}
\usepackage{amsmath}
\usepackage{color}
\usepackage{bm}
\usepackage{stackrel}
\usepackage{pgf,tikz}

\theoremstyle{plain}
\newtheorem{Thm}{Theorem}[section]
\newtheorem{Lem}[Thm]{Lemma}
\newtheorem{Prop}[Thm]{Proposition}
\newtheorem{Cor}[Thm]{Corollary}

\theoremstyle{definition}
\newtheorem{remark}[Thm]{Remark}

\numberwithin{equation}{section}

\newcommand{\tcr}[1]{\textcolor{red}{#1}}

\newcommand{\CC}{{\mathbb C}}
\newcommand{\NN}{{\mathbb N}}

\newcommand{\RR}{{\mathbb R}}

\DeclareMathOperator{\res}{res}
\DeclareMathOperator{\Span}{span}
\DeclareMathOperator{\cond}{cond}
\newcommand\im{\Im}
\newcommand\re{\Re}

\newcommand\dist{{\mathop{\mbox{\rm dist}}}}

\def\<{\langle}
\def\>{\rangle}

\begin{document}

\title{Zero free regions for Dirichlet series}

\author{Christophe Delaunay}
\author{Emmanuel Fricain}
\author{Elie Mosaki}
\author{Olivier Robert}
\address{Christophe Delaunay,Universit\'e de Lyon; Universit\'e Lyon 1; Institut Camille Jordan CNRS UMR 5208; 43, boulevard du 11 Novembre 1918, F-69622 Villeurbanne}
\email{delaunay@math.univ-lyon1.fr}
\address{Emmanuel Fricain, Universit\'e de Lyon; Universit\'e Lyon 1; Institut Camille Jordan CNRS UMR 5208; 43, boulevard du 11 Novembre 1918, F-69622 Villeurbanne}
\email{fricain@math.univ-lyon1.fr}
\address{Elie Mosaki,Universit\'e de Lyon; Universit\'e Lyon 1; Institut Camille Jordan CNRS UMR 5208; 43, boulevard du 11 Novembre 1918, F-69622 Villeurbanne}
\email{mosaki@math.univ-lyon1.fr}
\address{Olivier Robert
Universit\'e de Lyon, F-42023, Saint-Etienne, France;
Universit\'e de Saint-Etienne, F-42000, Saint-Etienne, France;
Laboratoire de math\'ematiques (LAMUSE, EA 3989)
23, rue du Dr P. Michelon
F-42000, Saint-Etienne France;}
\email{olivier.robert@univ-st-etienne.fr}
\date{}
\thanks{This work was supported by the ANR project no. 07-BLAN-0248 "ALGOL", the ANR project no. 09-BLAN-005801 "FRAB" and the ANR project no. 08-BLAN-0257 "PEPR".}



\begin{abstract} 
In this paper, we are interested in explicit zero-free discs for some Dirichlet series and we  also study a general Beurling-Nyman criterion for $L$-functions. 
Our results generalize and improve previous results obtained by N.~Nikolski  and by A.~de Roton. As a concrete application, we get, for example, a Beurling-Nyman type criterion for the Siegel zero problem.   
\end{abstract}

\maketitle

%
%
%
%
%
%
%
%

\section{Introduction}

In this article, we are interested in zero-free regions for  functions that are obtained as meromorphic continuation of Dirichlet series $s\mapsto \sum_{n\ge 1}a_nn^{-s}$. Such study arises naturally in various fields of mathematics such as functional analysis and number theory.

The particular case of the Riemann zeta function has been most studied, and is related to the Riemann Hypothesis, asserting that the zeta function does not vanish on the half-plane $\re(s)>1/2$. Several attempts have been made in the direction of solving or reformulating this conjecture.

In his thesis B. Nyman \cite{Nyman} gave  a reformulation of the Riemann Hypothesis by means of functional analysis. More precisely let $\{\cdot\}$ denotes the fractional part and let $\mathcal N$ be the set of functions 
\[
f(x)=\sum_{j=1}^n c_j \left\{\frac{\theta_j}{x}\right\},
\]
where $0<\theta_j\leq 1$, $c_j\in\CC$ and 
\begin{eqnarray}\label{eq:admissibilite-Nyman}
\sum_{j=1}^n c_j\theta_j=0.
\end{eqnarray}
Then Nyman proved that the Riemann Hypothesis holds if and only if the characteristic function $\chi_{(0,1)}$ of the interval $(0,1)$ belongs to the closure of $\mathcal N$ in $L^2(0,1)$. Later on A. Beurling \cite{Beurling} extended this result by proving that if $1<p<+\infty$, then the Riemann zeta function has no zeros in $\Re(s)>1/p$ if and only if $\chi_{(0,1)}$ belongs to the closure of $\mathcal N$ in $L^p(0,1)$. The case $p=1$ has been investigated in \cite{BF}. After the works of Beurling and Nyman, several results occur in this direction, see for example \cite{Nikolski-AIF}, \cite{balazard}, \cite{BBLS2000}, \cite{baez-duarte} and \cite{anne-TAMS}.


In particular, N. Nikolski \cite{Nikolski-AIF} gave, in a way, an effective version of the Beurling-Nyman criterion and  produces explicit zero-free regions for the Riemann zeta function: let $r>0$ and $\lambda\in\CC$ with $\Re(\lambda)>0$ be fixed parameters and let $\tilde{K}_r$ be the subspace of $L^2((0,1),dx/x)$ spanned by functions 
\[
E_{\alpha,r}(x)=x^r\left(\left\{\frac{\alpha}{x}\right\}-\alpha\left\{\frac{1}{x}\right\}\right),\qquad 0<x<1,
\] 
where $0\leq \alpha\leq 1$. Then the zero-free regions obtained by Nikolski are domains of the form 
\begin{equation}\label{nonzero0}
r+\left\{\mu\in\CC:\left|\frac{\mu-\lambda}{\mu+\bar\lambda}\right|^2<1-2\re(\lambda)\tilde{d_r}(\lambda)^2\right\},
\end{equation}
where $\tilde{d_r}(\lambda)=\dist(x^\lambda,\tilde{K}_r)$ is the distance  in $L^2((0,1),dx/x)$ between $x^\lambda$ and the subspace $\tilde{K}_r$. In the case $\lambda=r=1/2$, if $\tilde{d_r}(\lambda)=0$ the region \eqref{nonzero0} corresponds to the half-plane $\re(\mu)>1/2$ and Nikolski recovers Nyman's result (note that in this case, $\tilde{d}_{1/2}(1/2)=0$ if and only if $\chi_{(0,1)}$ belongs to the closure of $\mathcal N$ in $L^2(0,1)$; see Proposition~\ref{prop:ss-espace}).


More recently, A. de Roton \cite{anne-TAMS} generalized Nyman's work to the Selberg class of Dirichlet functions and reformulated the Generalized Riemann Hypothesis. Finally let us mention that in \cite{BBLS2000} and \cite{burnol} interesting lower bounds are obtained. First the authors proved in \cite{BBLS2000} that if $0<\lambda\leq 1$ and $\mathcal N^\sharp_\lambda$ is the closed span in $L^2(0,+\infty)$ of functions 
\[
f(x)=\sum_{j=1}^n c_j\left\{\frac{\theta_j}{x}\right\},\qquad x>0,
\]
where $\lambda<\theta_j\leq 1$, $c_j\in\CC$, then the Riemann Hypothesis holds if and only if $\lim_{\lambda\to 0}d(\lambda)=0$, where $d(\lambda)=\dist(\chi_{(0,1)},\mathcal N_\lambda^\sharp)$. In other words, we can drop the condition \eqref{eq:admissibilite-Nyman} in the Nyman's theorem. Furthermore they proved that 
\[
\liminf_{\lambda\to 0}d(\lambda)\sqrt{\log \frac{1}{|\lambda|}}>0.
\]
Then this lower bound was improved by J.F. Burnol in \cite{burnol} and generalized by de Roton in \cite{Anne-BSMF} and \cite{Anne-JNT}  for the Selberg class.
\bigskip

The aim of our article is to make a further investigation of Nikolski's work in a more general situation including, in particular, the Selberg class. Associated to some 
auxiliary function $\varphi \colon [0,\infty) \rightarrow \CC$, we introduce a class of spaces $K_r$  (in particular, in the case of the Riemann zeta function, $K_r$ 
contains $\tilde{K}_r$ for a special choice of $\varphi$). Then, for a large class of Dirichlet series, we obtain explicit zero-free regions that are larger than Nikolski's 
regions in the case of zeta. For some well chosen parameters, if $\dist(x^\lambda,K_r)=0$ our zero-free regions correspond to $\re(\mu)>r$ and we recover (and in fact
 improve a little bit) the results of de Roton concerning the reformulation of the Generalized Riemann Hypothesis. Finally, we investigate several explicit 
 applications. In particular, we give zero-free regions for the zeta function and Dirichlet $L$-functions ; it should be pointed out that the domains obtained have the  
 property of being explicit, but they do not have the ambition of competing with the classical non-zero regions that are involved, for instance, in the prime number 
 theorem. An other application we obtain is somehow a Beurling-Nyman's criterion for the Siegel's zero problem of Dirichlet $L$-functions. As far as we know, it seems
  to be a new  criterion concerning this question.  
\bigskip
~\\

%
%
%
%
%
%
%
%

\section{Notations and statements of the main results}
For a generic $s\in \CC$, we denote by $\sigma$ (respectively $t$) its real (respectively imaginary) part so that we have $s=\sigma+it$. For $s\in \CC$, we denote by $\Pi_s$ the half-plane defined by
$$
\Pi_s =\Pi_\sigma = \{ z \in \CC \; : \; \re(z)>\sigma \}.
$$
For the sequel, we fix a  Dirichlet series 
$$
L(s)=\sum_{n\geq 1}\frac{a_n}{n^s}
$$
satisfying the following conditions:
\begin{itemize}
\item For every $\varepsilon >0$, we have $a_n \ll_\varepsilon n^\varepsilon$. 
\item There exists $\sigma_0<1$ such that the function $s \mapsto L(s)$ admits a meromorphic continuation to $\re (s)>\sigma_0$ with a unique pole of order $m_L$ at $s=1$.
\item The function $s \mapsto (s-1)^{m_L} L(s)$ is analytic with finite order in $\Pi_{\sigma_0}$.
\end{itemize}
The growth condition on the coefficients $(a_n)_n$ implies that $L(s)$ is an absolutely convergent Dirichlet series for $\re(s)>1$. The second condition tells us that the function $(s-1)^{m_L}L(s)$ can be analytically continued in some half-plane $\Pi_{\sigma_0}$ which contains $\Pi_1$. Note that we do not require 
neither an Euler product nor a functional equation for $L(s)$. 
\\
We also consider an auxiliary function $\varphi:[0,+\infty[ \longrightarrow\CC$ satisfying the following conditions:
\begin{itemize}
\item $\varphi$ is supported on $[0,1]$.
\item $\varphi$ is locally bounded on $(0,1)$. 
\item $\varphi(x)=O(x^{-\sigma_0})$ when $x\to 0$.
\item $\varphi(x)=O((1-x)^{-\sigma_1})$ when $x\to 1^-$, for some $\sigma_1<1/2$.
\end{itemize}
The fact that $\sigma_1<1$ is sufficient for the integral 
$$
\int_0^1 \varphi(t) t^s \frac{dt}{t}
$$
to be  absolutely convergent for $\re(s)>\sigma_0$. Hence, the Mellin transform $\widehat{\varphi}$ of $\varphi$ is analytic on $\Pi_{\sigma_0}$. 
The condition $\sigma_1<1/2$ will become clearer in Lemma~\ref{lemme:intermediaire}. Recall that the (unnormalized) Mellin transform of a  Lebesgue-measurable function $\varphi:[0,+\infty[\to\CC$  is the function $\widehat\varphi$ defined by
$$
\widehat{\varphi} (s) = \int_0^{+\infty} \varphi(t) t^s \frac{dt}{t}  \qquad (s \in \CC),
$$ 
whenever the integral is absolutely convergent. Let $H^2(\Pi_\sigma)$ be the Hardy space of analytic functions $f$ : $\Pi_\sigma \rightarrow \CC$
such that $\Vert f \Vert_2 < \infty$ where
$$
\Vert f \Vert_2 = \sup_{x > \sigma} \left( \int_{-\infty}^{+\infty} \vert f(x +it)\vert^2 dt \right)^{\frac 12}.
$$
Then the (normalized) Mellin transform
$$
\begin{array}{cccc}
{\mathcal M} \; : \; & L^2_*\left((0,1), \frac{dt}{t^{1-2\sigma}}\right) &  \longrightarrow  & H^2(\Pi_\sigma) \\
		                  &  \varphi  										& \longmapsto       & \frac{1}{\sqrt{2\pi}} \widehat{\varphi}
\end{array}
$$
is a unitary operator (use the Paley-Wiener's theorem and the change of variables going from the Fourier transform to the Mellin transform,  for instance see \cite[p. 166] {nikolski2}). Here, we write $L^2_*\left((0,1), \frac{dt}{t^{1-2\sigma}}\right)$ for the subspace of functions in $L^2\left((0,+\infty), \frac{dt}{t^{1-2\sigma}}\right)$  that vanish almost everywhere on $(1,+\infty)$.
We define
\[
\psi(u)=\res\left(L(s)\hat\varphi(s)u^s,s=1\right)-\sum_{n <u }a_n\varphi\left(\frac nu\right) \qquad (u\in\RR_+),
\]
where $\res(F(s),s=1)$ denotes the residue of the meromorphic function $F$ at $s=1$.
The method that is going to be explored depends on the fact that the function $\psi$ belongs to $L^2((1,+\infty), \frac{du}{u^{1+2r}})$ for a certain real number $r>\sigma_0$. The following gives a criterion for this property.
\begin{Thm}\label{Thm:cns-integrabilite-psi} 
Let $r>\sigma_{0}$. If $m_{L} \geq 1$, we assume furthermore that $r\neq 1$. The following are equivalent:
\begin{enumerate}
\item[$\mathrm{(i)}$] The function $\psi$ belongs to $L^2((1,+\infty), \frac{du}{u^{1+2r}})$.
\item[$\mathrm{(ii)}$] The function $t\longmapsto L(r+it)\hat\varphi(r+it)$ belongs to $L^2(\RR)$.
\end{enumerate}
\end{Thm}

We will see in Corollary~\ref{Cor:estimation-norme-psi-r-plusgrand1} that the conditions $\mathrm{(i)}$ and $\mathrm{(ii)}$ above are satisfied whenever $r>1$. Furthermore, Theorem \ref{Thm:cns-integrabilite-psi} is a generalization of \cite[Proposition 3.3]{anne-TAMS} when $\varphi = \chi_{(0,1)}$ and $r=1/2$.

We see that the second condition of the theorem above depends on some growth estimates of $L$ in vertical strip which is, in its turn, linked with questions related to the convexity bound  and to the Lindel\"of hypothesis. In particular, if $L$ is a function in the Selberg class, then one can prove that $L$ satisfies the generalized Lindel\"of hypothesis if and only if for every $k\in\mathbb N$, we have
\begin{equation}\label{eq:lindelof-hypothesis}
t\longmapsto \frac{L^k(\frac 12+it)}{\frac12+it}\in L^2(\RR).
\end{equation}
Note that with the special choice of $\varphi=\chi_{(0,1)}$, then $\hat\varphi(s)=1/s$ and the condition $\mathrm{(ii)}$ above for $r=1/2$ means exactly that \eqref{eq:lindelof-hypothesis} is satisfied for $k=1$. Moreover, in \cite{anne-TAMS} using the functional equation, it is shown that the condition $\psi\in L^2((1,+\infty),\frac{du}{u^2})$ (in the case when $\varphi=\chi_{(0,1)}$) is necessary for the generalized Riemann Hypothesis for $L$-functions in the Selberg class. In \cite{Anne-Acta-Arithm}, it is also shown that the condition on $\psi$ (still with $\varphi=\chi_{(0,1)}$) is satisfied for $L$-functions in the Selberg class of degree less than $4$.
\medskip
\\
Fix an integer $m\geq 0$ and let $W=\bigcup_{n\geq 1} (0,1]^n$. We say that $\alpha \in W$ is of length $n$ if $\alpha$ belongs to $(0,1]^n$. For each $\alpha$ in $W$, its length is denoted by $\ell(\alpha)$. Now let $\alpha\in W$ and $c\in\CC^{\ell(\alpha)}$, we say that $A=(\alpha,c)$ is an $m$-admissible sequence if  
\begin{eqnarray}\label{eq:admissibilite-general}
\sum_{j=1}^{\ell(\alpha)} c_j \alpha_j (\log \alpha_j)^k =0 \mbox{ for all } 0\leq k \leq m-1 \; .
\end{eqnarray}
Furthermore, $A$ is said to be non-trivial if $c\in \CC^{\ell(\alpha)}\setminus\{(0,\dots,0)\}$.
%
%

It is easy to prove (see Lemma~\ref{Lem:cle-zeros-communs} and Lemma~\ref{lemma_admissible}) that for any fixed $\ell>m$, there are  infinitely many non-trivial $m$-admissible sequences of length $\ell$. We also easily see that every $A=(\alpha,c)$ is $0$-admissible. Note that the notion of  admissible  sequences had been introduced in \cite{anne-TAMS} and it is a generalization of condition~\eqref{eq:admissibilite-Nyman} which appears in Nyman's theorem and which corresponds 
to ~$m=1$.\bigskip
~\\
{\bf From now on, we assume that $r$ is chosen such that $\psi \in L^2((1,+\infty), \frac{du}{u^{1+2r}})$.}\bigskip
~\\
Then we associate to each $m_L$-admissible sequence $A=(\alpha,c)$ the function $f_{A,r}$ defined by
$$
f_{A,r}(t) = t^{r-\sigma_0} \sum_{j=1}^{\ell(\alpha)} c_j \psi\left(\frac{\alpha_j}{t}\right) , \qquad t >0.
$$
We will show that $f_{A,r} \in L^2_*((0,1), \frac{dt}{t^{1-2\sigma_0}})$ and we let
$$
K_r = \Span\{ f_{A,r} \; : \; A \mbox{~a~}m_L\mbox{-admissible~sequence} \}
$$
and 
\[
d_r(\lambda)=\dist\left(t^{\bar\lambda-2\sigma_0}\chi_{(0,1)},K_r\right) \qquad (\lambda\in\Pi_{\sigma_0}),
\]
where the (closed) span and the distance are taken with respect to the space $L^2_*((0,1), \frac{dt}{t^{1-2\sigma_0}})$.  Let us remark that we trivially have $d_r^2(\lambda)\leq 1/(2\re(\lambda)-2\sigma_0)$.

One of our  main theorems is the following which gives zero-free discs for the function~$L$.
\begin{Thm}\label{Thm:disque-sans-zeros}
Let $\lambda\in\Pi_{\sigma_0}$. Then $L$ does not vanish on $r-\sigma_0+D_r(\lambda)$, where
\[
D_r(\lambda):=\left\{\mu\in\CC:\left|\frac{\mu-\lambda}{\mu+\bar\lambda-2\sigma_0}\right|<\sqrt{1-2(\re(\lambda)-\sigma_0)d^2_r(\lambda)}\right\}.
\]
\end{Thm}
Note that the zero-free regions which appear in Theorem~\ref{Thm:disque-sans-zeros} are discs (empty if $d_r^2(\lambda)=1/(2\re(\lambda)-2\sigma_0)$) or half-planes. More precisely, for $\lambda=a+ib\in\Pi_{\sigma_0}$ ($a>\sigma_0,b\in\RR$) and $R\in [0,1]$, then the set
\[
\left\{\mu\in\CC:\left|\frac{\mu-\lambda}{\mu+\bar\lambda-2\sigma_0}\right|<R\right\}
\]
is the open (euclidean) disc whose center is $\Omega=\left(\frac{a+R^2(a-2\sigma_0)}{1-R^2},b\right)$ and radius is $\frac{2R(a-\sigma_0)}{1-R^2}$ if $R\in [0,1[$; if $R=1$ this set is the half-plane $\Pi_{\sigma_0}$. In both cases, we easily see that this set is contained in the half-plane $\Pi_{\sigma_0}$.  

As a corollary of the proof of Theorem \ref{Thm:disque-sans-zeros}, we obtain an other explicit version.
 \begin{Cor}\label{Cor2:disque-sans-zeros}
Let $\lambda\in\Pi_{\sigma_0}$. Then $L$ does not vanish on the disc
\[
r-\sigma_0+\left\{\mu\in\CC:\left|\frac{\mu-\lambda}{\mu+\bar\lambda-2\sigma_0}\right|<\sqrt{2(\re(\lambda)-\sigma_0)}\frac{|\widehat{f_{A,r}}(\lambda)|}{\|f_{A,r}\|_2}\right\},
\]
for any $m_L$-admissible sequence $A$. 
\end{Cor}
Note that taking $L(s)=\zeta(s)$, $\ell(\alpha)=2$ and $\varphi= \chi_{(0,1)}$ (so that $\sigma_0=\sigma_1=0$) we recover exactly the results of Nikolski (\cite{Nikolski-AIF}). Now, taking $\varphi(t)=(1-t)^{-\sigma_1}\chi_{(0,1)}(t)$, we obtain larger zero-free discs  whenever $\im(\lambda)$ is large enough. We refer the reader to Section \ref{examples} for further discussions about the Rieman zeta function and other examples. 

We will see in Theorem~\ref{Thm:cle-transforme-mellin} that 
$$
\widehat{f_{A,r}}(\lambda)= L(\lambda+r-\sigma_0)\hat\varphi(\lambda+r-\sigma_0)\left(\sum_{j=1}^n c_j \alpha_j^{\lambda+r-\sigma_0}\right), \qquad \lambda \in \Pi_{\sigma_0}.
$$ 
Hence Corollary~\ref{Cor2:disque-sans-zeros} can be understood as  follows: let $\lambda\in\Pi_{\sigma_0}$ such that $L(\lambda+r-\sigma_0) \neq 0$; then there is a small neighborhood  of $\lambda+r-\sigma_0$ free of zeros for  $L$. Of course, this is an obvious consequence of the continuity of the function $L$ but the interest of Corollary~\ref{Cor2:disque-sans-zeros} is that it gives an explicit neighborhood where the function $L$ does not vanish; moreover, this explicit neighborhood is expressed in terms of the values of $L$ and in particular it does not use any estimates of the derivatives.

Finally, we obtain a general Beurling-Nyman type theorem.
\begin{Thm}\label{Thm:reformulation-HRG}
Suppose that the function $\hat\varphi$ does not vanish on the half-plane $\Pi_{r}$, that 
$\limsup_{x\to+\infty}\frac{\log|\hat\varphi(x+r-\sigma_0)|}{x} = 0$ and that $a_1 \neq 0$. Then the following assertions are equivalent:
\begin{enumerate}
 \item The function $L$  does not vanish on the half-plane $\Pi_{r}$.
\item There exists $\lambda\in\Pi_{\sigma_0}$ such that  $d_r(\lambda)=0$.
\item For all $\lambda\in\Pi_{\sigma_0}$, we have $d_r(\lambda)=0$.
\item We have $K_r = L^2_*((0,1),dt/t^{1-2\sigma_0})$.
\end{enumerate}
\end{Thm}
Taking $L(s)=\zeta(s)$, $\ell(\alpha)=2$, $r=\lambda=1/2$ and $\varphi=\chi_{(0,1)}$ (so that we can take $\sigma_0=\sigma_1=0$), we obtain Nyman's theorem. Now, taking $\varphi(t) = (1-t)^{-\sigma_1} \chi_{(0,1)}(t)$ we extend the results of \cite{anne-TAMS} (see also Section \ref{examples} for more details). 

An other consequence of Theorem \ref{Thm:disque-sans-zeros} is, in a way, a Beurling-Nyman criterion for Dirichlet $L$-functions. More precisely, let $\chi$ be a Dirichlet character with conductor $q$ and $L(\chi,s)$ its $L$-function. Then, for $1/2 \leq r < 1$, we define $d_r$ by 
$$
d_r=\min_{\ell, c, \alpha} \left( \int_0^1 \left| t^{1-r} - t^r\sum_{j=1}^{\ell}  c_j \sum_{n<\alpha_j/t} \chi(n)\right|^2 \frac{dt}{t} \right)^{\frac{1}{2}}
$$
where the minimum is taken over all $\ell \geq 0$, $c=(c_j) \in \CC^\ell$ and $\alpha=(\alpha_j) \in (0,1]^\ell$. One can show (see Proposition~\ref{Prop:dr-dirichlet}) that $d_r^2 < 1/(2-2r)$. We have 
\begin{Thm}\label{dirichlet} If 
$$
d_r^2 \leq \frac{1}{2-2r} - \frac{C^2}{2 (\log q)^2 (1-r)^3}
$$
for some (absolute) constant $C$ and some $1/2 \leq r \leq 1$, then $L(\chi,\sigma)$ does not vanish in the real-interval $\sigma >1- C/\log q$.
\end{Thm}
In order to obtain the criterion for the Siegel zero problem, then we consider all Dirichlet characters $\chi$ and an absolute constant $C$ independant of $\chi$.

The next section is devoted to the proof of Theorem \ref{Thm:cns-integrabilite-psi} and to the study of the function $\psi$.  Section \ref{admissible} will focus on  the admissible sequences and the functions $f_{A,r}$.
Theorem \ref{Thm:disque-sans-zeros} and Corollary \ref{Cor2:disque-sans-zeros} will be proven in Section~\ref{Zeros_free_regions} and Theorem \ref{Thm:reformulation-HRG} will be proven in Section~\ref{sec:BN-theorems}. Some explicit examples will be studied in Section \ref{examples} in which we will prove Theorem \ref{dirichlet}.

%
%
%
%
%
%
%
%

\section{The function $\psi$ and proof of theorem \ref{Thm:cns-integrabilite-psi}}
We define the functions $\psi_{1}$ and $\psi_{2}$ by
\begin{eqnarray*}
\psi_{1}(u) &=& \res\left(L(s)\hat\varphi(s)u^s,s=1\right) \qquad (u\in\RR_+),\\
\psi_{2}(u) &=& \sum_{n <u }a_n\varphi\left(\frac nu\right) \qquad (u\in\RR_+),
\end{eqnarray*} 
so that $\psi(u)=\psi_{1}(u) - \psi_{2}(u)$ by definition.

The function $s \mapsto \hat\varphi(s)$ is analytic in $\Pi_{\sigma_{0}}$ and the meromorphic continuation of $L(s)$ has a pole (of order $m_L$) only at 
$s=1$, hence we can write
\begin{equation}\label{eq:decomposition-F-varphichapeau}
L(s)\hat\varphi(s)=\sum_{k=1}^{m_L}\frac{p_{-k}}{(s-1)^k}-H(s),
\end{equation}
with $H$ analytic in $\Pi_{\sigma_{0}}$.
\begin{Lem}\label{lemme crucial} We have 
\begin{equation}\label{eq:transformee-Mellin-fonction-H}
H(s)=\int_0^1\psi\left(\frac 1t\right)t^{s-1}\,dt  \,=\hat\phi(t), \;\;\;\;\; \re(s)>1,
\end{equation}
where $\phi(t)=\psi(1/t)\chi_{(0,1)}(t)$.
\end{Lem}
\begin{proof} [{\it Proof of Lemma \ref{lemme crucial}}] On the one hand, we have
\begin{equation}\label{mellin psi 2}
L(s) \hat{\varphi}(s)= \int_{1}^{+\infty} \psi_{2}(u) u^{-s-1} du,\qquad \re(s)>1.
\end{equation}
Indeed, this equality comes from the following computation and  Fubini's theorem:
\begin{eqnarray*}
L(s)\hat{\varphi}(s) &=& \left( \sum_{n\geq 1} \frac{a_{n}}{n^{s}} \right) \int_{0}^{1} \varphi(t) t^{s-1} dt = \sum_{n\geq 1} a_n \int_0^1 \varphi(t) \left(\frac{t}{n}\right)^s \frac{dt}{t} \\
					 &=& \sum_{n\geq 1} a_n \int_n^{+\infty} \varphi\left(\frac{n}{u}\right) u^{-s-1} du \\
					 &=& \int_1^{+\infty} \sum_{n<u} a_n \varphi\left(\frac{n}{u}\right) u^{-s-1} du \\
					 &=& \int_1^{+\infty} \psi_2(u) u^{-s-1} du.
\end{eqnarray*}
Note that Fubini's theorem can be applied here because $\sigma=\re(s)>1$ and then 
\[
\sum_{n\geq 1}\frac{|a_n|}{n^{\sigma}}\int_0^1 |\varphi(t)|t^{\sigma-1}\,dt<+\infty.
\]
On the other hand, for each $1\le k\le m_L$, we have
\begin{equation}\label{residu}
\res\left(\frac{u^s}{(s-1)^k},s=1\right)=\frac{u(\log u)^{k-1}}{(k-1)!},
\end{equation}
and an easy induction argument gives
$$
\frac{1}{(k-1)!}\int_1^{+\infty}(\log u)^{k-1}u^{-s}du=\frac{1}{(s-1)^k}.
$$

Hence, with the notation of \eqref{eq:decomposition-F-varphichapeau}, the equality
\begin{equation}\label{mellin psi 1}
\int_1^{+\infty}\psi_{1}(u)u^{-s-1}\,du=\sum_{k=1}^{m_L}\frac{p_{-k}}{(s-1)^k}\,.
\end{equation}
follows by linearity.

Equations (\ref{mellin psi 2}) and (\ref{mellin psi 1}) imply
$$
H(s)=\sum_{k=1}^{m_L}\frac{p_{-k}}{(s-1)^k}-L(s)\hat\varphi(s)=\int_1^{+\infty}\psi(u)u^{-s-1}\,du=\int_0^1\psi\left(\frac 1t\right)t^{s-1}\,dt.
$$
\end{proof}
\begin{remark}\label{poly} From (\ref{eq:decomposition-F-varphichapeau}) and  (\ref{residu}), we easily see that we can write $\psi_1(t)=tP(\log t)$ where $P$ is a polynomial of degree $<m_L$ ($P\equiv 0$ if $m_L=0$). More precisely we have
\[
\psi_1(t)=t\sum_{k=1}^{m_L}\frac{p_{-k}}{(k-1)!}(\log t)^{k-1}.
\]
\end{remark}
\begin{Lem}\label{lemme:intermediaire} The function $s \mapsto H(s)$ is of finite order on $\Pi_r$. Moreover, for all $\sigma>1$, the function $t\mapsto H(\sigma + i t)$ belongs to $L^2(\RR)$ and
$$
\int_\RR |H(\sigma + it)|^2 dt = O\left(\frac{1}{\sigma^{1-2\sigma_1}} \right) \; \; \mbox{ for } \sigma \rightarrow + \infty.
$$
\end{Lem}
\begin{proof} [{\it Proof of Lemma \ref{lemme:intermediaire}}]
The function $s\mapsto H(s)$ is bounded on some neighborhood $V_1$ of $s=1$. The functions $s\mapsto \sum_j\frac{p_{-j}}{(s-1)^j}$
and (by assumption) $s \mapsto L(s)$ are of finite order on $\Pi_{\sigma_0}\smallsetminus V_1$. Moreover, $s\mapsto \hat{\varphi}(s)$ is bounded 
on the closure of $\Pi_r$ since
$$
|\hat{\varphi}(s)|\le \int_0^1|\varphi(x)|x^{r}\frac{dx}{x},\qquad \re(s)\geq r,
$$
and the last integral is convergent because $r>\sigma_0$ and $\varphi(x)=O(x^{-\sigma_0} (1-x)^{-\sigma_1})$. Hence we can conclude that $H$ is of finite order on $\Pi_r$, which proves the first part of the lemma.\smallskip
\\
Let $\sigma>1$ be a fixed real number. We easily check that, for every $1\leq j\leq m_L$, the function $t\longmapsto (\sigma-1+it)^{-j}$ belongs to $L^2(\RR)$ and we have
\begin{eqnarray}\label{eq:estimation-triviale-fraction-rationnelle}
\int_\RR\left|\frac{1}{(\sigma+it-1)^j} \right|^2\,dt=O\left(\frac{1}{\sigma^{2j-1}}\right),\qquad \sigma\to+\infty.
\end{eqnarray}
On the other hand, using Plancherel's formula and the fact that $\sigma_1<1/2$ and $\sigma>1>\sigma_0$, we have 
\begin{eqnarray}\label{eq:estimation-triviale2}
\int_{\RR}|\hat{\varphi}(\sigma+it)|^2dt=2\pi \int_{0}^1|\varphi(x)|^2 x^{2\sigma}\frac{dx}{x}<+\infty.
\end{eqnarray}
More precisely 
\begin{eqnarray}\label{eq:estimation-triviale3}
\int_{0}^1|\varphi(x)|^2 x^{2\sigma}\frac{dx}{x}=O\left(\frac{1}{\sigma^{1-2\sigma_1}}\right),\qquad \sigma\to+\infty.
\end{eqnarray}
Indeed, we have $\varphi(x) = O(x^{-\sigma_0} (1-x)^{-\sigma_1})$ and
$$
\int_0^1 |x^{-\sigma_0} (1-x)^{-\sigma_1}|^2 x^{2\sigma}  \frac{dx}{x} = \beta(2(\sigma-\sigma_0) ,1-2\sigma_1)
$$
where $\beta(s,z)$ is the beta function. Using Stirling's formula, we have 
\[
\beta(2(\sigma-\sigma_0),1-2\sigma_1)\sim \frac{\Gamma(1-2\sigma_1)}{(2\sigma)^{1-2\sigma_1}},\qquad \sigma\to +\infty.
\]
So we get the estimate~\eqref{eq:estimation-triviale3}. 

Therefore $t\longmapsto \hat\varphi(\sigma+it)$ belongs to $L^2(\RR)$ and we have
\begin{equation}\label{estimate}
\int_{\RR}|\hat{\varphi}(\sigma+it)|^2dt=O\left(\frac{1}{\sigma^{1-2\sigma_1}}\right),\qquad \sigma\to +\infty. 
\end{equation}
It remains to notice that $s\mapsto L(s)$ is bounded in $\Pi_{\sigma}$ (since by hypothesis $\sigma>1$) and then the second part of the lemma follows immediately from \eqref{eq:estimation-triviale-fraction-rationnelle}, \eqref{estimate} and \eqref{eq:decomposition-F-varphichapeau}.
\end{proof}

\noindent
Now we can  prove Theorem \ref{Thm:cns-integrabilite-psi}.\medskip 
~\\
Assume that $\mathrm{(i)}$ is satisfied. Then  the function $\phi(t)=\psi\left(\frac 1t\right)\chi_{(0,1)}(t)$  belongs to $L^2_*((0,1),\frac{dt}{t^{1-2r}})$ and thus the function 
$G:=\hat\phi$ belongs to $H^2(\Pi_{r})$. The analytic continuation principle implies that the equality \eqref{eq:transformee-Mellin-fonction-H} is satisfied for every $s\in\Pi_{r}$, that is $H(s)=\hat\phi(s)=G(s)$, $s\in\Pi_{r}$. Since $G\in H^2(\Pi_{r})$, we know that the function $G^*$, defined by
\[
G^*(t):=\lim_{\sigma \stackrel[>]{}{\to} r} G(\sigma+it),
\]
exists almost everywhere on $\RR$ and belongs to $L^2(\RR)$. But 
\[
G^*(t)=\lim_{\sigma \stackrel[>]{}{\to} r}H(\sigma+it)=H(r+it),
\]
because $H$ is continuous on the closed half-plane $\re(s)\geq r$. Therefore $t\longmapsto H(r+it)$ belongs to $L^2(\RR)$. It remains to notice that for every $1\leq k\leq m_L$, the function $t\longmapsto (r-1+it)^{-k}$ belongs to $L^2(\RR)$ (this is where we have to assume that $r\neq 1$ whenever $m_L\geq 1$). Thus according to \eqref{eq:decomposition-F-varphichapeau}, we get that 
$t\mapsto L(r+it)\hat\varphi(r+it)$ belongs to $L^2(\RR)$.\\
\\
Conversely assume that $\mathrm{(ii)}$ is satisfied. Then the function $t\longmapsto H(r+it)$ belongs to $L^2(\RR)$. 

Let $\sigma_2 > \max(1,r)$. The function $H$ is analytic on $\Omega:=\{s:\thinspace r<\re(s)<\sigma_2\}$ and continuous on the closure of $\Omega$. 
Using the Hardy, Littlewood and Ingham's theorem \cite[Theorem 7]{HIP}, we deduce from Lemma~\ref{lemme:intermediaire} that
$$
\sup_{r\le \sigma\le \sigma_2}\int_{\RR}|H(\sigma+it)|^2dt<+\infty.
$$
Hence, the second part of Lemma~\ref{lemme:intermediaire} gives that $H$ belongs to $H^{2}(\Pi_{r})$. Since we have $H^2(\Pi_{r})=\mathcal M\big(L^2_*\left((0,1),\frac{dt}{t^{1-2r}}\right)\big)$, there exists $\phi_1\in L^2_*\left((0,1),\frac{dt}{t^{1-2r}}\right)$ such that $H(s)=\hat\phi_1(s)$, for every $s\in \Pi_{r}$. Thus for $\re(s)>\max(1,r)$, we have $\hat\phi_1(s)=H(s)=\hat \phi(s)$.

By injectivity of the Mellin transform, we get that 
\[
\psi\left(\frac 1t\right)\chi_{(0,1)}(t)=\phi(t)=\phi_1(t).
\]
Thus $t\longmapsto \psi\left(\frac 1t\right)\chi_{(0,1)}(t)$ belongs to $L^2_*\left((0,1),\frac{dt}{t^{1-2r}}\right)$, which implies that $\psi$ belongs to $L^2\left((1,+\infty),\frac{du}{u^{1+2r}}\right)$ and that concludes the proof of Theorem \ref{Thm:cns-integrabilite-psi}.

\hfill $\square$

\begin{Cor}\label{Cor:estimation-norme-psi-r-plusgrand1} Let $r>1$. Then $\psi \in L^2\left((1,+\infty), \frac{du}{u^{1+2r}}\right)$ and we have
$$
\Vert \psi \Vert_2 = O\left(r^{\sigma_1-1/2}\right), \mbox{ as } r \rightarrow +\infty.
$$
\end{Cor}

\begin{proof}
Let $r>1$ be a fixed real number.  By Lemma~\ref{lemme:intermediaire}, the function $t\longmapsto H(r+it)$ belongs to $L^2(\RR)$ and for every $1\leq j\leq m_L$, the function $t\longmapsto (r-1+it)^{-j}$ belongs also to $L^2(\RR)$. Therefore it follows from \eqref{eq:decomposition-F-varphichapeau} that $t\longmapsto L(r+it)\hat\varphi(r+it)$ belongs to $L^2(\RR)$ and Theorem~\ref{Thm:cns-integrabilite-psi}  implies that $\psi\in L^2((1,+\infty),\frac{du}{u^{1+2r}})$. Moreover, if we let as before $\phi(t)=\psi(1/t)\chi_{(0,1)}(t)$, we have by Plancherel's formula
\begin{eqnarray*}
\|\psi\|_2^2&=&\int_1^{+\infty}|\psi(u)|^2\frac{du}{u^{1+2r}}=\int_0^1\left|\psi\left(\frac{1}{t}\right)\right|^2\,\frac{dt}{t^{1-2r}}\\
&=&\|\phi\|^2_{L^2_*((0,1),\frac{dt}{t^{1-2r}})}\\
&=&\frac{1}{2\pi}\int_\RR |\hat\phi(r+it)|^2\,dt\\
&=&\frac{1}{2\pi}\int_\RR |H(r+it)|^2\,dt,
\end{eqnarray*}
and Lemma~\ref{lemme:intermediaire} gives the result. 
\end{proof}

\begin{remark}
In fact,  if $L(s)$ has some nice arithmetical properties (e.g. satisfying  a Wiener-Ikehara type theorem), one can expect that the main contributions of $\psi_1$ will be compensated by the main contributions of $\psi_2$ and so that the function $\psi$ will belong to the space $L^2\left((1,+\infty), \frac{du}{u^{1+2r}}\right)$ for smaller value of $r$. For example, if $L(s)=\zeta(s)$  and if $\varphi(t)=\chi_{0,1}(t)$ for all $t\in (0,1)$, then $\psi(u)=u - \lceil u \rceil +1 $ (so $\psi(u)=\{u \}$ for almost all $u\in (1,\infty)$) and $\psi \in L^2\left((1,+\infty), \frac{du}{u^{1+2r}}\right)$ for all $r>0$. 
In this case $\hat{\varphi}(s)=1/s$ and $t \mapsto |\zeta(r+it)|/|r+it| \in L^2(\RR)$ for all $r>0$ as expected by Theorem \ref{Thm:cns-integrabilite-psi}.
We will discuss about other examples in Section~\ref{examples}.
\end{remark}

\begin{remark}
Assume that $\varphi$ is bounded at $t=1$ (which corresponds to $\sigma_1=0$). Thus, for every $\alpha>1$, we have $\psi(t)=\psi_1(t)-\psi_2(t)=O(t^\alpha)$, $t\to+\infty$. Indeed, on the one hand, as we have seen $\psi_1(t)=tP(\log t)$, where $P$ is a polynomial of degree $<m_L$. On the other hand, let $\alpha>1$ be a fixed real number. Then there exists a constant $C>0$ such that 
\[
|\varphi(t)|\leq C t^{-\alpha},\qquad t\in (0,1).
\] 
Hence we get
\begin{eqnarray*}
|\psi_2(t)|&\leq&\sum_{n<t}|a_n|\left|\varphi\left(\frac{n}{t}\right)\right| \leq C \sum_{n<t}|a_n|\left(\frac{t}{n}\right)^{\alpha}\\
&\leq& C't^{\alpha}\sum_{n<t}\frac{1}{n^{\alpha-\varepsilon}},
\end{eqnarray*}
which gives that $\psi_2(t)=O(t^{\alpha})$  as $t\to+\infty$. It is now easy to check that $\psi$ belongs to  $L^2((1,+\infty),\frac{du}{u^{1+2r}})$, for every $r>1$ and $\|\psi\|_2=O(r^{-1/2})$, $r\to +\infty$. We recover Corollary~\ref{Cor:estimation-norme-psi-r-plusgrand1}.
\end{remark}

%
%
%
%
%
%
%
%

\section{Admissible sequences and the functions $f_{A,r}$}\label{admissible}

The following results will be useful in the sequel. 
\begin{Lem}\label{lemma_admissible}
Let $(\alpha,c)$ be an $m$-admissible sequence.
\begin{enumerate}
\item For any polynomial $P$ of degree $d<m$, we have
$$
\sum_{j=1}^{\ell(\alpha)} c_j \alpha_j P(\log \alpha_j) =0.
$$
\item For any positive real number $\lambda_1$ and for any real number $\lambda_2\geq \max_j(\alpha_j)$, both $((\alpha_j^{\lambda_1})_j,(c_j\alpha_j^{1-\lambda_1})_j)$ and $((\alpha_j/\lambda_2)_j,c)$ are $m$-admissible sequences.
\end{enumerate}
\end{Lem}
The proof is easy and follows immediately from the definition and so we omit it.
\noindent
For an $m$-admissible sequence $A=(\alpha,c)$ we define the entire function
\[
g_A(s)=\sum_{j=1}^{\ell(\alpha)}c_j\alpha_j^{s}\qquad (s\in \CC).
\]
Using \eqref{eq:admissibilite-general}, we notice that these functions satisfy
\begin{equation}\label{derivees}
g_A^{(k)}(1)=0\qquad (0\le  k \le m-1).
\end{equation}
\begin{Lem}\label{Lem:cle-zeros-communs} Let $m$ be a nonnegative integer. Then  we have the followings
\begin{enumerate}
\item For every integer $\ell\geq m+1$ and every $0<\alpha_1<\alpha_2 <\cdots <\alpha_\ell \leq 1$, there exists $c\in\CC^{\ell}$ such that $A=(\alpha,c)$ is an $m$-admissible sequence and $g_A^{(m)}(1)\neq 0$. Furthermore we can choose $c_\ell \neq 0$. 
\item For every $s_1\in \CC\smallsetminus\{1\}$, there exists an $m$-admissible sequence $A$ such that $g_A(s_1)\neq 0$.
\end{enumerate}
\end{Lem}

\begin{proof}
$(1)$  Let $\ell \geq m+1$, then there is a unit vector $(b_1,b_2,\dots,b_m,b_\ell)\in\CC^{m+1}$ such that 
\begin{eqnarray}\label{eq:equation-vandermonde}
\left(
\begin{array}{ccccc}
1&\dots&\dots& 1&1\\
(\log\alpha_1)&\dots&\dots&(\log\alpha_{m})&(\log\alpha_\ell)\\
\vdots&       &     & \vdots &\vdots \\
(\log\alpha_1)^m&\dots&\dots&(\log\alpha_{m})^m&(\log\alpha_\ell)^m\\
\end{array}
\right)
\left(
\begin{array}{c}
b_1\\
\vdots\\
b_{m}\\
b_\ell\\
\end{array}
\right)
=
\left(
\begin{array}{c}
0\\
\vdots\\
0\\
1\\
\end{array}
\right),
\end{eqnarray}
since the corresponding van der Monde matrix is invertible.
Setting 
\[
c_j=\begin{cases}
\frac{b_j}{\alpha_j},&1\leq j\leq m\\
0,&m+1\leq j\leq \ell-1\\
\frac{b_\ell}{\alpha_\ell},&j=\ell.
\end{cases}
\]
Then $A=(\alpha,c)$ is an $m$-admissible sequence and $g_A^{(m)}(1)=1$. Furthermore we easily see that $b_\ell\neq0$. Indeed if $b_\ell=0$ then we have 
\[
\left(
\begin{array}{cccc}
1&\dots&\dots& 1\\
(\log\alpha_1)&\dots&\dots&(\log\alpha_{m})\\
\vdots&       &     & \vdots \\
(\log\alpha_1)^{m-1}&\dots&\dots&(\log\alpha_{m})^{m-1}\\
\end{array}
\right)
\left(
\begin{array}{c}
b_1\\
\vdots\\
\vdots\\
b_{m}\\
\end{array}
\right)
=
\left(
\begin{array}{c}
0\\
\vdots\\
\vdots\\
0\\
\end{array}
\right),
\]
and using once more that the corresponding van der Monde matrix is invertible, we get $b_1=b_2=\dots=b_m=b_\ell=0$, which contradicts \eqref{eq:equation-vandermonde}. Therefore $b_\ell\not=0$ and then~$c_\ell\not=0$. 
\medskip
\\ 
(2) Let $A=(\alpha,c)$ be a non trivial $m$-admissible sequence. For $0<\lambda<1$, consider $g_{A_\lambda}(s)=g_A(1-\lambda+\lambda s)$. Note that
$$
g_{A_\lambda}(s)=\sum_{j=1}^{\ell}c_j\alpha_j^{1-\lambda}(\alpha_j^{\lambda})^{s},
$$ 
and, by Lemma~\ref{lemma_admissible}, $A_\lambda=((\alpha_j^{\lambda})_j,(c_j\alpha_j^{1-\lambda})_j)$ is an $m$-admissible sequence. Now let $s_1\neq 1$; then there necessarily exists $0<\lambda<1$ such that  $g_{A_\lambda} (s_1)\neq 0$, since $g_A$ is analytic and non identically zero.
\end{proof}
\begin{Thm}\label{Thm:cle-transforme-mellin}
Let $r>\sigma_0$ such that $\psi \in L^2((1,+\infty),\frac{du}{u^{1+2r}})$ and let $A=(\alpha,c)$ be an $m_L$-admissible sequence.
We define
$$
f_{A,r}(t)= t^{r-\sigma_0} \sum_{j=1}^{\ell(\alpha)} c_j \psi\left(\frac{\alpha_j}{t}\right) \qquad (t>0).
$$
Then we have:
\begin{enumerate}
\item $f_{A,r}(t)=0$, if $t> \max_j \alpha_j$.
\item $f_{A,r}(t)\in L^2_*((0,1),\frac{dt}{t^{1-2\sigma_0}})$.
\item For $\re(s)>\sigma_0$, we have
\begin{eqnarray}\label{eq:transformee-mellin-fcalphagamma}
\widehat{f_{A,r}}(s)=-L(s+r-\sigma_0)\hat\varphi(s+r-\sigma_0)g_{A}(s+r-\sigma_0).
\end{eqnarray}
\end{enumerate}
\end{Thm}
\begin{proof} 
(1) As before, we write $\psi = \psi_1 - \psi_2$. From the remark \ref{poly}, we have $\psi_1(u)=uP(\log u)$ where $P$ is a polynomial of degree $<m_L$. Then
\begin{eqnarray*}
 \sum_{j=1}^{\ell(\alpha)} c_j \psi\left(\frac{\alpha_j}{t}\right) &=&  \sum_{j=1}^{\ell(\alpha)} c_j \psi_1\left(\frac{\alpha_j}{t}\right) -  \sum_{j=1}^{\ell(\alpha)} c_j \psi_2\left(\frac{\alpha_j}{t}\right) \\
 &=& \sum_{j=1}^{\ell(\alpha)} c_j \frac{\alpha_j}{t} P\left(\log(\frac{\alpha_j}{t})\right) - \sum_{j=1}^{\ell(\alpha)} c_j \sum_{n <\alpha_j/t} a_n \varphi\left( \frac{nt}{\alpha_j}\right).
\end{eqnarray*}
If $t>\max_j (\alpha_j)$ the second sum is 0 because $\alpha_j/t<1$. Furthermore, by Lemma \ref{lemma_admissible}, $(\alpha/t,c)$ is an $m_L$-admissible sequence and the first sum is also zero.\medskip 
\\
(2) Each of the terms  $t^{r-\sigma_0} \psi\left(\frac{\alpha_j}{t}\right)$ belongs to the space $L^2((0,1),\frac{dt}{t^{1-2\sigma_0}})$. Indeed, if $\Vert \alpha \Vert :=\max_j \alpha_j$, we have
\begin{eqnarray*}
\int_0^{\Vert \alpha \Vert}  t^{2r-2\sigma_0}\left|\psi\left(\frac{\alpha_j}{t}\right)\right|^2 \frac{dt}{t^{1-2\sigma_0}} &=& \int_0^{\Vert \alpha \Vert} t^{2r-1} \left| \psi\left(\frac{\alpha_j}{t}\right)\right|^2 dt \\
&=& \alpha_j^{2r} \int_{\frac{\alpha_j}{\Vert \alpha \Vert}}^{+\infty} |\psi(u)|^2 \frac{du}{u^{1+2r}} \\
&=& \alpha_j^{2r} \int_{\frac{\alpha_j}{\Vert \alpha \Vert}}^1 |\psi(u)|^2 \frac{du}{u^{1+2r}} + \alpha_j^{2r} \int_1^{+\infty} |\psi(u)|^2 \frac{du}{u^{1+2r}}\\
&=&\alpha_j^{2r} \int_{\frac{\alpha_j}{\Vert \alpha \Vert}}^1 |\psi_1(u)|^2 \frac{du}{u^{1+2r}} + \alpha_j^{2r} \int_1^{+\infty} |\psi(u)|^2 \frac{du}{u^{1+2r}},
\end{eqnarray*}
since $\psi(u)=\psi_1(u)$ for $u<1$. Now, by hypothesis, the second integral is finite and the first integral is also finite because $\psi_1(u)=uP(\log u)$ and $\frac{\alpha_j}{\Vert \alpha \Vert}>0$.
Furthermore, we see that
\begin{equation}\label{remarque_norme}
\Vert f_{A,r} \Vert_{L^2_*((0,1),\frac{dt}{t^{1-2\sigma_0}})} \leq \sum_{j=1}^{\ell(\alpha)} |c_j\alpha_j^r|\left(\left( \int_{\frac{\min_j \alpha_j}{\max_j \alpha_j}}^1 \left|\psi_1(u)\right|^2 \frac{du}{u^{1+2r}}\right)^{1/2} + \Vert \psi \Vert_{L^2((1,+\infty),\frac{du}{u^{1+2r}})}\right).
\end{equation}
(3) For $\re(s)>1$, we have by (\ref{mellin psi 2})
$$
L(s) \hat{\varphi}(s) = \int_1^{+\infty} \psi_2(u) u^{-s-1} du.
$$
Hence
\begin{eqnarray*}
L(s) \hat{\varphi}(s) \sum_{j=1}^{\ell(\alpha)} c_j \alpha_j^s &=& \sum_{j=1}^{\ell(\alpha)} c_j \int_1^{+\infty} \left(\frac{u}{\alpha_j}\right)^{-s} \psi_2(u) \frac{du}{u} \\
&=& \sum_{j=1}^{\ell(\alpha)} c_j \int_0^{\alpha_j} \psi_2\left(\frac{\alpha_j}{t}\right) \frac{dt}{t^{1-s}} \\
&=& \sum_{j=1}^{\ell(\alpha)} c_j \int_0^1 \psi_2\left(\frac{\alpha_j}{t}\right) \frac{dt}{t^{1-s}},
\end{eqnarray*}
the last equality follows from $\psi_2(\alpha_j/t)=0$ if $\alpha_j < t \leq 1$. 

Using once more Lemma~\ref{lemma_admissible}, we have 
\[
\sum_{j=1}^{\ell(\alpha)} c_j \psi\left(\frac{\alpha_j}{t}\right)=-\sum_{j=1}^{\ell(\alpha)} c_j \psi_2\left(\frac{\alpha_j}{t}\right), 
\]
whence
\begin{eqnarray*}
L(s) \hat{\varphi}(s) \sum_{j=1}^{\ell(\alpha)} c_j \alpha_j^s &=&-\int_0^1 \left(\sum_{j=1}^{\ell(\alpha)} c_j \psi\left(\frac{\alpha_j}{t}\right)\right)\,\frac{dt}{t^{1-s}}\\
&=&-\int_0^1 t^{\sigma_0-r}f_{A,r}(t)\frac{dt}{t^{1-s}}\\
&=&-\int_0^1 f_{A,r}(t)t^{s+\sigma_0-r-1}\,dt\\
&=&-\widehat f_{A,r}(s+\sigma_0-r),\qquad \re(s)>1.
\end{eqnarray*}
So, for $\re(s) >1+\sigma_0-r$, we have $\widehat{f_{A,r}}(s)=-L(s+r-\sigma_0)\hat\varphi(s+r-\sigma_0)g_{A}(s+r-\sigma_0)$. By the analytic continuation principle, the equality holds for $\re(s)>\sigma_0$. (Note that from \eqref{derivees} the pole of $L(s+r-\sigma_0)$  is killed by the zero of $g_A(s+r-\sigma_0)$ at $s=1-r+\sigma_0$.)
\end{proof}

%
%
%
%
%
%
%
%

\section{Zeros free regions for Dirichlet series}\label{Zeros_free_regions}

Before proving Theorem \ref{Thm:disque-sans-zeros} we need some well known tools concerning the Hardy space $H^2(\Pi_{\sigma_0})$. For these we refer to \cite[Chapter 8]{Hoffman}. Actually, in  \cite{Hoffman}, the following facts are stated for the space $H^2(\Pi_0)$ but it is easy to obtain the corresponding results for $H^2(\Pi_{\sigma_0})$, for instance using the unitary map $h(z)\longmapsto h(z-\sigma_0)$ from $H^2(\Pi_{\sigma_0})$ onto $H^2(\Pi_0)$.

Recall that if $h\in H^2(\Pi_{\sigma_0})$, then
\[
h^*(\sigma_0+it):=\lim_{\substack{\sigma\to\sigma_0\\
>}}h(\sigma+it)
\]
exists for almost every $t\in\RR$ (with respect to the Lebesgue measure). Moreover we have $h^*\in L^2(\sigma_0+i\RR)$ and $\|h\|_2=\|h^*\|_2$.  We can therefore identify (unitarely) $H^2(\Pi_{\sigma_0})$ with a (closed) subspace of $L^2(\sigma_0+i\RR)$. In the following we use the symbol $h$  not only for the function in $H^2(\Pi_{\sigma_0})$ but also for its "radial " limit (in other words we forget the star). This identification enables us to consider $H^2(\Pi_{\sigma_0})$ as an Hilbert space, with scalar product given by
 \begin{eqnarray}\label{eq:definition-produit-scalaire-H2}
\langle h,g\rangle_2=\int_{-\infty}^{+\infty}h(\sigma_0+it)\overline{g(\sigma_0+it)}\,dt,\qquad h,g\in H^2(\Pi_{\sigma_0}).
\end{eqnarray}
Now for $\lambda\in\Pi_{\sigma_0}$, we have the following integral representation
\[
h(\lambda)=\frac{1}{2\pi}\int_{-\infty}^{+\infty}\frac{h(\sigma_0+it)}{\lambda-\sigma_0-it}\,dt,
\]
so that we can write, using \eqref{eq:definition-produit-scalaire-H2}, 
\begin{eqnarray}\label{eq1:noyau-reproduisant-H2}
h(\lambda)=\langle h,k_\lambda\rangle_2,
\end{eqnarray}
where $k_\lambda$ is the function in $H^2(\Pi_{\sigma_0})$ defined by
\begin{eqnarray}\label{eq2:noyau-reproduisant-H2}
k_\lambda(z):=\frac{1}{2\pi}\frac{1}{z-2\sigma_0+\bar\lambda},\qquad z\in\Pi_{\sigma_0}.
\end{eqnarray}
The function $k_\lambda$ is called the reproducing kernel of $H^2(\Pi_{\sigma_0})$ and we have
\begin{eqnarray}\label{eq3:noyau-reproduisant-H2}
\|k_\lambda\|_2^2=\langle k_\lambda,k_\lambda\rangle_2=k_\lambda(\lambda)=\frac{1}{4\pi(\re(\lambda)-\sigma_0)}.
\end{eqnarray}

Recall now a  useful property of factorization for functions in the Hardy space.
\begin{Lem}\label{Lem:elementaire-cle-espace-Hardy}
Let $h\in H^2(\Pi_{\sigma_0})$ and $\mu\in\Pi_{\sigma_0}$ such that $h(\mu)=0$. Then there exists $g\in H^2(\Pi_{\sigma_0})$ such that $\|h\|_2=\|g\|_2$ and
\[
h(z)=\frac{z-\mu}{z+\bar\mu-2\sigma_0}g(z),\qquad z\in\Pi_{\sigma_0}.
\] 
\end{Lem}
\begin{proof}
See the results in \cite[Chapter 8, pp. 132]{Hoffman} and apply the transform $h(z)\longmapsto h(z-\sigma_0)$.

\end{proof}

\noindent
{\bf Proof of  Theorem \ref{Thm:disque-sans-zeros}} Denote by $E_r=\mathcal M K_r$. Since the Mellin transform is a unitary map from $L^2_*((0,1),\frac{dt}{t^{1-2\sigma_0}})$ onto $H^2(\Pi_{\sigma_0})$,  we have 
\[
E_r=\hbox{span }_{H^2(\Pi_{\sigma_0})}\left(h_{A,r} \; : \; A \hbox{ a $m_L$-admissible sequence}\right),
\]
where $h_{A,r}(s)=\mathcal Mf_{A,r}(s)$, for $\re(s)>\sigma_0$. It follows from Theorem~\ref{Thm:cle-transforme-mellin} that 
$$
h_{A,r}(s)=-\frac{1}{\sqrt{2\pi}}\,L(s+r-\sigma_0)\hat\varphi(s+r-\sigma_0)g_{A}(s+r-\sigma_0), \;\; \re(s)>\sigma_0.
$$
Now assume that there is $\mu\in \Pi_{\sigma_0}$ such that $L(\mu+r-\sigma_0)=0$. Since $\mu\in\Pi_{\sigma_0}$, we get that $h_{A,r}(\mu)=0$, for every $m_L$-admissible sequence $A$. Thus, for all $h\in E_r$ with $\|h\|_2=1$, we have $h(\mu)=0$. Since $h\in H^2(\Pi_{\sigma_0})$, we know from Lemma~\ref{Lem:elementaire-cle-espace-Hardy} that there is $g\in H^2(\Pi_{\sigma_0})$ such that $\|g\|_2=\|h\|_2=1$ and
\[
h(z)=\frac{z-\mu}{z+\bar\mu-2\sigma_0}g(z) \qquad (z\in\Pi_{\sigma_0}).
\]
Hence with Cauchy-Schwarz inequality and \eqref{eq1:noyau-reproduisant-H2}, we deduce that
\[
|h(\lambda)|=\left|\frac{\lambda-\mu}{\lambda+\bar\mu-2\sigma_0}\right||g(\lambda)|\leq \left|\frac{\lambda-\mu}{\lambda+\bar\mu-2\sigma_0}\right| \|g\|_2 \|k_\lambda\|_2 = \left|\frac{\lambda-\mu}{\lambda+\bar\mu-2\sigma_0}\right| \|k_\lambda\|_2,
\]
so
\[
\sup_{\substack{h\in E_r\\ \|h\|_2=1}}|h(\lambda)|\leq\left|\frac{\mu-\lambda}{\mu+\bar\lambda-2\sigma_0}\right| \|k_\lambda\|_2.
\]
By contraposition, we have proved that $L$ does not vanish on
$$
r-\sigma_0 +\left\{\mu\in\CC:\left|\frac{\mu-\lambda}{\mu+\bar\lambda-2\sigma_0}\right|<\frac{\sup_{\substack{f\in E_r\\ \|h\|_2=1}}|h(\lambda)|}{\|k_\lambda\|_2} \right\}.
$$
It remains to prove that
$$
\frac{\sup_{{\substack{h\in E_r\\ \|h\|_2=1}}}|h(\lambda)|^2}{\|k_\lambda\|_2^2} = 1-2(\re(\lambda)-\sigma_0)d^2_r(\lambda).
$$
To show this equality, first remark that
\[
\sup_{\substack{h\in E_r\\ \|h\|_2=1}}|h(\lambda)|=\sup_{\substack{h\in E_r \\ \|h\|_2=1}}|\langle h,k_\lambda\rangle_2|=\|P_{E_r}k_\lambda\|_2,
\]
where $P_{E}$ denotes the orthogonal projection onto the closed subspace $E$ of $H^2(\Pi_{\sigma_0})$. Then Pythagoras' Theorem implies that
\[
\sup_{\substack{h\in E_r\\
\|h\|_2=1}}|h(\lambda)|^2=\|P_{E_r}k_\lambda\|^2_2=\|k_\lambda\|_2^2-\|P_{E_r^\perp}k_\lambda\|_2^2.
\]
Hence
$$
\frac{\sup_{{\substack{h\in E_r\\ \|h\|_2=1}}}|h(\lambda)|^2}{\|k_\lambda\|_2^2} = 1-\frac{\hbox{dist}^2(k_\lambda,E_r)}{\|k_\lambda\|_2^2}.
$$
But for $\re(s)>\sigma_0$, we have 
\[
\mathcal M\left(t^{\bar\lambda-2\sigma_0}\chi_{(0,1)}(t)\right)(s)=\frac{1}{\sqrt{2\pi}}\int_0^1 t^{\bar\lambda-2\sigma_0}t^{s-1}\,dt=\frac{1}{\sqrt{2\pi}}\frac{1}{\bar\lambda-2\sigma_0+s}=\sqrt{2\pi}k_\lambda(s),
\]
according to \eqref{eq2:noyau-reproduisant-H2}. Since the Mellin transform is an isometry from $L^2_*((0,1),\frac{dt}{t^{1-2\sigma_0}})$ onto $H^2(\Pi_{\sigma_0})$, we obtain
\[
\hbox{dist}(k_\lambda,E_r)=\frac{1}{\sqrt{2\pi}}\hbox{dist}(t^{\bar\lambda -2\sigma_0}\chi_{(0,1)}(t)  ,K_r)=\frac{d_r(\lambda)}{\sqrt{2\pi}}.
\]
We conclude the proof of Theorem~\ref{Thm:disque-sans-zeros} using $\|k_\lambda\|_2^2 = \frac{1}{4\pi(\re(\lambda) - \sigma_0)}$.

\hfill $\square$
\\
\noindent
{\bf Proof of Corollary \ref{Cor2:disque-sans-zeros}} 
It follows from the proof of Theorem~\ref{Thm:disque-sans-zeros} and  \eqref{eq3:noyau-reproduisant-H2} that $L$ does not vanish on
\begin{eqnarray*}
r-\sigma_0+\left\{\mu\in\CC:\left|\frac{\mu-\lambda}{\mu+\bar\lambda-2\sigma_0}\right|<\sqrt{4\pi(\re(\lambda)-\sigma_0)}\sup_{\substack{h\in E_r \\
\|h\|_2=1}}|h(\lambda)|\right\}.
\end{eqnarray*}
So in particular (with $h=h_{A,r}$), we get that $L$ does not vanish on 
\[
r-\sigma_0+\left\{\mu\in\CC:\left|\frac{\mu-\lambda}{\mu+\bar\lambda-2\sigma_0}\right|<\sqrt{4\pi(\re(\lambda)-\sigma_0)}\frac{|h_{A,r}(\lambda)|}{\|h_{A,r}\|_2}\right\},
\]
which proves corollary \ref{Cor2:disque-sans-zeros} because $h_{A,r}={\mathcal M}f_{A,r} = \frac{1}{\sqrt{2\pi}}\widehat{f_{A,r}}$.\hfill $\square$

\bigskip
The following could be interesting in applications.
\begin{Cor}\label{Cor1:disque-sans-zeros}
Let $M$ be a subspace of $K_r$ and $\lambda \in \Pi_{\sigma_0}$. Then $L$ does not vanish on the disc
\[
r-\sigma_0+\left\{\mu\in\CC:\left|\frac{\mu-\lambda}{\mu+\bar\lambda-2\sigma_0}\right|^2<1-2(\re(\lambda)-\sigma_0) \dist^2(t^{\bar\lambda-2\sigma_0}\chi_{(0,1)},M)\right\}.
\]
\end{Cor}

\begin{proof}
It is sufficient to note that $\hbox{dist}(t^{\bar\lambda-2\sigma_0}\chi_{(0,1)},M)\geq\hbox{dist}(t^{\bar\lambda-2\sigma_0}\chi_{(0,1)},K_r)=d_r(\lambda)$ and then apply Theorem~\ref{Thm:disque-sans-zeros}.
\end{proof}
Adapting the proof of Theorem \ref{Thm:disque-sans-zeros}, we could obtain immediately the following generalization.
\begin{Thm}  Let $k\in \NN$ and $\lambda\in\Pi_{\sigma_0}$.  The function $L$ does not have any zero of order greater or equal to $k$ on 
\[
r-\sigma_0 + \left\{\mu\in\CC:\left|\frac{\mu-\lambda}{\mu+\bar\lambda-2\sigma_0}\right|< \left(1-2(\re(\lambda)-\sigma_0)d^2_r(\lambda)\right)^{\frac{1}{2k}}\right\}.
\]
\end{Thm}

%
%
%
%
%
%
%
%

\section{A Beurling--Nyman type theorem for Dirichlet series}\label{sec:BN-theorems}
One of the main steps for proving Theorem \ref{Thm:reformulation-HRG} is to prove that the space $E_r$ is a closed subspace of $H^2(\Pi_{\sigma_0})$ that is invariant under multiplication operator $\tau_v$, $v\geq 0$, and then to apply Lax--Beurling's theorem. Before recalling this theorem, we give some notations and results. We refer to \cite[Part A, Chap. 2 \& 6]{nikolski2} for more details. 

For $v \in \RR$, let $\tau_v$ be the operator of multiplication on $L^2(\sigma_0+i\RR)$ defined by
\[
(\tau_v f)(\sigma_0+it)=e^{-ivt}f(\sigma_0+it) \qquad (f\in L^2(\sigma_0+i\RR)).
\]
We also denote by $H^\infty(\Pi_{\sigma_0})$ the Hardy space of bounded analytic functions on $\Pi_{\sigma_0}$; as in $H^2(\Pi_{\sigma_0})$, functions in $H^\infty(\Pi_{\sigma_0})$ admit "radial" limits at almost every points of the boundary $\re(s)=\sigma_0$ and if $f\in H^\infty(\Pi_{\sigma_0})$ and $f^*$ is its boundary limits, then  $f^*\in L^\infty(\sigma_0+i\RR)$ and $\|f^*\|_\infty=\|f\|_\infty$ ($=\sup_{z\in\Pi_{\sigma_0}}|f(z)|$). A function $\Theta$ in $H^\infty(\Pi_{\sigma_0})$ is said to be inner if $|\Theta(\sigma_0+it)|=1$ for almost every point $t\in\RR$.

It is easy to see that if $\Theta$ is inner and  $E=\Theta H^2(\Pi_{\sigma_0})$, then $E$ is a closed subspace of $H^2(\Pi_{\sigma_0})$ invariant by $\tau_v$, $v\geq 0$. Lax--Beurling's theorem gives the converse.
\begin{Thm}[Lax--Beurling]\label{Thm:Beurling}
Let $E$ be a closed subspace of $H^2(\Pi_{\sigma_0})$ such that $\tau_v E\subset E$, $\forall v\geq 0$. Then there is an inner function $\Theta \in H^{\infty}(\Pi_{\sigma_0})$ unique (up to a constant of modulus one) such that 
$E=\Theta H^2(\Pi_{\sigma_0})$.
\end{Thm}
\begin{remark}\label{rem:beurling}
When $E$ is spanned by a family of functions, we can precise a little bit more the conclusion of Theorem~\ref{Thm:Beurling}. Indeed, let $E$ be a closed subspace of $H^2(\Pi_{\sigma_0})$ spanned by a family of functions $(f_i)_{i\in I}$, $f_i\in H^2(\Pi_{\sigma_0})$ and  let $f_i=h_i g_i$ be the factorization of $f_i$ in inner factor $h_i$ and outer factor $g_i$. If $\tau_v E\subset E$, $\forall v\geq 0$, then $E=\Theta H^2(\Pi_{\sigma_0})$, where $\Theta=\hbox{gcd}(h_i:i\in I)$ is the greatest common inner divisor of the family $(h_i)_{i\in I}$. 
\end{remark}
~\\
{\bf Proof of Theorem  \ref{Thm:reformulation-HRG}} First of all, note that  
\begin{eqnarray}\label{eq:zeros-communs}
\bigcap_{h\in E_r}h^{-1}(\{0\})=\{s\in\Pi_{\sigma_0}:L(s+r-\sigma_0)=0\},
\end{eqnarray}
where we recall that 
\[
E_r=\mathcal M K_r=\hbox{span }_{H^2(\Pi_{\sigma_0})}\left(h_{A,r}: A \hbox{  a  $m_L$-admissible sequence}\right),
\]
and 
\[
h_{A,r}(s)=-\frac{1}{\sqrt{2\pi}}\,L(s+r-\sigma_0)\hat\varphi(s+r-\sigma_0)g_{A}(s+r-\sigma_0), \;\; \re(s)>\sigma_0.
\]
Indeed, the first inclusion 
\[
\{s\in\Pi_{\sigma_0}:L(s+r-\sigma_0)=0\}\subset\bigcap_{h\in E_r}h^{-1}(\{0\})
\]
is trivial. For the converse inclusion, it is sufficient to notice that according to Lemma~\ref{Lem:cle-zeros-communs}, the only common zeros to functions 
$g_{A}(s+r-\sigma_0)$ is $s=1-r+\sigma_0$ but this zero is compensated by the singularities of $L(s+r-\sigma_0)$ at this point.

The proofs of $(4)$ $\Longrightarrow$ $(3)$ and $(3)\Longrightarrow (2)$  are trivial.

$(2)\Longrightarrow (1)$: let $\lambda\in\Pi_{\sigma_0}$ such that  $d_r(\lambda)=0$.  Then according to Theorem~\ref{Thm:disque-sans-zeros}, the function $L$ does not vanish on
\[
r-\sigma_0+\left\{\mu\in\CC:\left|\frac{\mu-\lambda}{\mu+\bar\lambda-2\sigma_0}\right|<1\right\},
\]
and it follows from the paragraph after Theorem~\ref{Thm:disque-sans-zeros} that this region is precisely the half-plane  $\Pi_{r}$.

$(1)\Longrightarrow (4)$: denote by $S_\beta$ the "shift" on $L^2_*((0,1),\frac{dt}{t^{1-2\sigma_0}})$ defined by 
\[
(S_\beta f)(t):=\beta^{-\sigma_0}f\left(\frac{t}{\beta}\right), \qquad 0<\beta\leq 1,0\leq t\leq 1.
\]
It is clear that $S_\beta$ is a unitary operator on $L^2_*((0,1),\frac{dt}{t^{1-2\sigma_0}})$. We will show that
\begin{eqnarray}\label{eq:invariance-K-gamma-shift}
S_\beta K_r\subset K_r \qquad (0<\beta\leq 1).
\end{eqnarray}
If $A=(\alpha,c)$ is a $m_L$-admissible sequence and $0\leq t\leq 1$, we have 
\begin{align*}
(S_\beta f_{A,r})(t)=&\beta^{-\sigma_0}f_{A,r}\left(\frac{t}{\beta}\right)\\
=&\beta^{-\sigma_0}\frac{t^{r-\sigma_0}}{\beta^{r-\sigma_0}}\sum_{j=1}^{\ell(\alpha)}c_j\psi_k\left(\frac{\alpha_j\beta}{t}\right)\\
=&\beta^{-r} f_{A',r}(t),
\end{align*}
where $A'=(\alpha',c)$ is the $m_L$-admissible sequence with $\alpha'=(\alpha'_j)_{j}=(\beta\alpha_j)_j$ and so $f_{A',r}\in K_r$. Hence for every admissible sequence 
$A$, we have proved that $S_\beta f_{A,r}\in K_r$. Since $S_\beta$ is a bounded operator and the functions $f_{A,r}$ span the subspace $K_r$, we deduce \eqref{eq:invariance-K-gamma-shift}. Therefore we obtain that $E_r=\mathcal M K_r$ is a closed subspace of $H^2(\Pi_{\sigma_0})$ which is invariant under the semi-group of operators  $\mathcal M S_\beta \mathcal M^{-1}$, $0<\beta\leq 1$. Now let us show that
\begin{eqnarray}\label{eq:Sbeta-unitairement-equivalent-taubeta}
 \mathcal M S_\beta\mathcal M^{-1}=\tau_v,
\end{eqnarray}
with $v=-\log\beta$.
If $h\in H^2(\Pi_{\sigma_0})$ and if $f\in L^2_*((0,1),\frac{dt}{t^{1-2\sigma_0}})$ is such that $\mathcal M^{-1}h=f$, we have for $\re(s)>\sigma_0$,
\begin{align*}
(\mathcal M S_\beta \mathcal M^{-1}h)(s)=&\frac{1}{\sqrt{2\pi}}\int_0^{\infty}(S_\beta f)(t)t^{s-1}\,dt\\
=&\frac{1}{\sqrt{2\pi}}\int_0^{+\infty}\beta^{-\sigma_0}f\left(\frac{t}{\beta}\right)t^{s-1}\,dt\\
=&\beta^{s-\sigma_0}\frac{1}{\sqrt{2\pi}}\int_0^{+\infty}f(u)u^{s-1}\,du\\
=&\beta^{s-\sigma_0} (\mathcal M f)(s)=\beta^{s-\sigma_0} h(s).
\end{align*}
Hence we have 
\[
(\mathcal M S_\beta \mathcal M^{-1}h)(\sigma_0+it)=\beta^{it}h(\sigma_0+it)=e^{it\log\beta}h(\sigma_0+it)=(\tau_v h)(\sigma_0+it),
\]
with $v=-\log\beta$, which proves \eqref{eq:Sbeta-unitairement-equivalent-taubeta}. We therefore obtain that $\tau_v E_r\subset E_r$, for all $v\geq 0$ and then Lax--Beurling's Theorem (see Theorem~\ref{Thm:Beurling}) implies that there is an inner function $\Theta$ in the half-plane $\Pi_{\sigma_0}$ such that 
$E_r=\Theta H^2(\Pi_{\sigma_0})$. Moreover we know that $\Theta=BS$, where $B$ is a Blaschke product and $S$ is a singular inner function (\cite{nikolski2}). But according to Remark~\ref{rem:beurling}, the zeros of $B$ coincide with the set of common zeros of functions $h\in E_r$. Hence it follows from \eqref{eq:zeros-communs} and the hypothesis that $B$ has no zeros. In other words, $B\equiv 1$. We will now show that $S\equiv 1$. First note that since $h_{A,r}$ is analytic on $\re(s)>2\sigma_0-r$ and since $S$ is a common inner divisor of all functions in $E_r$, it follows that $S$ can be continued analytically through the axis $\sigma_0+i\RR$. In particular, this forces $S$ to be of the form $S(s)=e^{-a(s-\sigma_0)}$, for some $a\geq 0$ (see for instance \cite[Part A, Chap. 4 \& 6]{nikolski2}). Now let $h_{A,r}$ be a function in $E_r$ and write $h_{A,r}=S h$, with $h\in H^2(\Pi_{\sigma_0})$. 
\begin{Lem}\label{Lem:residu-logarithmique}
Let $h\in H^2(\Pi_{\sigma_0})$. Then 
\[
\limsup_{\sigma\to+\infty}\frac{\log |h(\sigma)|}{\sigma}\leq 0.
\]
\end{Lem}
\begin{proof} [{\it Proof of lemma \ref{Lem:residu-logarithmique}}] Using \eqref{eq1:noyau-reproduisant-H2} and \eqref{eq3:noyau-reproduisant-H2}, we get 
\[
|h(\sigma)|\leq \|h\|_2\|k_\sigma\|_2=\frac{\|h\|_2}{2\sqrt\pi (\sigma-\sigma_0)^{1/2}},
\]
for every $\sigma>\sigma_0$. Hence 
\[
\frac{\log |h(\sigma)|}{\sigma}\leq \frac{O(1)}{\sigma}-\frac{1}{2}\frac{\log(\sigma-\sigma_0)}{\sigma},
\]
which gives the result letting $\sigma \rightarrow +\infty$.
\end{proof}
~\\
On the one hand, according to the previous lemma, we have
\begin{eqnarray}\label{eq:123}
\limsup_{x\to+\infty}\frac{\log|h_{A,r}(x)|}{x}\leq\limsup_{x\to+\infty}\frac{\log|S(x)|}{x}=-a\leq 0.
\end{eqnarray}
And on the other hand, writing 
\[
h_{A,r}(x)=L(x+r-\sigma_0)\hat\varphi(x+r-\sigma_0)\sum_{j=1}^{\ell(\alpha)}c_j \alpha_j^{x+r-\sigma_0},
\]
we can assume, by Lemma \ref{Lem:cle-zeros-communs}, that $0<\alpha_1<\cdots <\alpha_{\ell(\alpha)}=1$ and that $c_{\ell(\alpha)}\neq 0$. Then we have
\[
|h_{A,r}(x)|\sim |a_1| |\hat\varphi(x+r-\sigma_0) c_{\ell(\alpha)}|,\qquad x\to+\infty.
\]
Hence we can find $x_0>1$ such that 
\[
|h_{A,r}(x)|\geq \frac{|a_1|}{2}\ |\hat\varphi(x+r-\sigma_0) c_{\ell(\alpha)}|,\qquad (x>x_0),
\]
which gives
\[
\log |h_{A,r}(x)|\geq -\log 2+\log|a_1|+\log|c_{\ell(\alpha)}| +\log|\hat\varphi(x+r-\sigma_0)| ,\; (x>x_0).
\]
Therefore we obtain that
\[
\limsup_{x\to +\infty} \frac{\log|h_{A,r}(x)|}{x}\geq \limsup_{x\to+\infty}\frac{\log|\hat\varphi(x+r-\sigma_0)|}{x} = 0.
\]
Using \eqref{eq:123}, we conclude that $a=0$. Hence $S\equiv 1$, which gives $E_r=H^2(\Pi_{\sigma_0})$, that is $K_r=L^2_*((0,1),\frac{dt}{t^{1-2\sigma_0}})$.  \hfill$\square$
\begin{remark} In fact, with the hypothesis on $\varphi$, one can easily show that we always have $ \limsup_{x\to+\infty}\frac{\log|\hat\varphi(x+r-\sigma_0)|}{x} \leq 0$. Note that there exists some $\varphi$ for which the $\limsup$ is negative.
\end{remark}

%
%
%
%
%
%
%
%

\section{Some examples and applications}\label{examples}
\subsection{The Riemann zeta function}
Let 
\[
\zeta(s)=\sum_{n\geq 1}\frac{1}{n^s} \qquad (\re(s)>1),
\]
be the Riemann zeta function. Then it is well known  that $\zeta$ can be meromorphically continued in the whole plane $\CC$, with  a unique pole of order $1$ at the point $s=1$ (\cite{titchmarsh}).  Thus, the function $\zeta$ satisfies our hypothesis with $m_L=1$ (and for instance, $\sigma_0=0$).  Now let us consider the function $\varphi$ defined on $[0,+\infty[$ by 
\[
\varphi(t)=\begin{cases}
(1-t)^{-\sigma_1} &\hbox{if }0\leq t <  1\\
0&\hbox{if }t\geq 1
\end{cases} ,
\]
where $\sigma_1 <1/2$ is fixed. 

Then an elementary computation shows that $\hat\varphi(s)=\frac{\Gamma(s) \Gamma(1-\sigma_1)}{\Gamma(s+1-\sigma_1)}$, $\re(s)>0$. Hence, 
\begin{equation}\label{stirling}
|\hat{\varphi}(s)| \sim \frac{\Gamma(1-\sigma_1)}{|t|^{1-\sigma_1}} \;\;\; \mbox{ as } |t|=|\Im(s)| \rightarrow \infty  .
\end{equation}
We now illustrate\footnote{It is an illustration since here Theorem \ref{Thm:cns-integrabilite-psi} is useless as it can be seen in Theorem \ref{pasbesoin}.} the use of Theorem \ref{Thm:cns-integrabilite-psi} in order to determine the values of $r$ such that
$$
\psi \in L^2\left((1,\infty), \frac{du}{u^{1+2r}}\right).
$$
An obvious computation shows that 
$$
\psi(u)= \frac{u}{1-\sigma_1} - \sum_{n<u}\frac{1}{\left(1-\frac{n}{u}\right)^{\sigma_1}}, \qquad u>0.
$$
Let $\mu=\mu(1/2)$ be a convexity bound for $\zeta(1/2+it)$ (\cite{tenenbaum}), so that we have for all~$\varepsilon >0$,
$$
|\zeta(r+it)| =\begin{cases}O_\varepsilon\left(|t|^{\frac12-(1-2\mu)r+\varepsilon}\right) & \mbox{ if } 0\leq r \leq 1/2, \\
O_\varepsilon\left(|t|^{2\mu(1-r)+\varepsilon}\right) & \mbox{ if } 1/2\leq r <1.
\end{cases} 
$$ 
We know that $\mu<1/4$. We get from \eqref{stirling}
$$
\left|\zeta(r+it)\hat{\varphi}(r+it)\right|^2 = \begin{cases}
O_\varepsilon \left(|t|^{-1-(2-4\mu)r+ 2\sigma_1+\varepsilon}\right)  & \mbox{ if } 0\leq r \leq 1/2, \\ 
O_\varepsilon \left(|t|^{-2+2\sigma_1+4\mu(1-r)+\varepsilon}\right)  & \mbox{ if } 1/2\leq r <1. \\ 
\end{cases} 
$$
So a direct application of Theorem \ref{Thm:cns-integrabilite-psi} gives:
\begin{Prop} With the notation above, then $\psi \in L^2\left((1,\infty), \frac{du}{u^{1+2r}}\right)$ if one the following holds:
\begin{itemize}
\item[$\bullet$] $\max(0,\sigma_1/(1-2\mu)) <r \leq 1/2$;
\item[$\bullet$] $\max(1/2,1-(1-2\sigma_1)/(4\mu)) < r <1$;
\item[$\bullet$] $r>1$.
\end{itemize}
\end{Prop}
By a result of Huxley (\cite{huxley}), we can take $\mu=\frac{32}{205}$ .
In particular, one may have  $r=1/2$ if $\sigma_1$ is chosen such that $\sigma_1<\frac{141}{410} \approx 0.343\dots$.

If we assume the Lindel\"of Hypothesis, we have $\mu(1/2)=0$ and a similar computation as above implies that if $r\neq 1$ and $r>\max(0,\sigma_1)$, then 
$$
\psi \in L^2\left((1,\infty), \frac{du}{u^{1+2r}}\right).
$$

In fact, the work above is useless for the Riemann zeta function since we can prove directly the following:
\begin{Thm}\label{pasbesoin} Let $r\neq 1$. Then  $\psi\in  L^2((1,\infty), \frac{du}{u^{1+2r}})$ if and only if $r>\max(0,\sigma_1)$. Moreover, if the condition holds we have
$$
\Vert \psi \Vert_{L^2((1,\infty),\frac{du}{u^{1+2r}})}^2  \leq C(\sigma_1) \zeta(1+2(r-\sigma_1)),
$$
where $C(\sigma_1) =  \frac{2^{3-2\sigma_1}}{(3-2\sigma_1)(1-\sigma_1)^2}+\frac{2^{2-\sigma_1}}{(1-\sigma_1)^2}+\frac{1}{1-2\sigma_1}$.
\end{Thm}
\begin{proof}
In the case of the zeta function and for the  previous choice of $\varphi$, we have
$$
\psi(u)=\frac{u}{1-\sigma_1}-\sum_{n<u}\left(1-\frac{n}{u}\right)^{-\sigma_1}\qquad (u>0).
$$
Hence
\begin{eqnarray*}
\Vert \psi \Vert_2^2 &=& \int_1^\infty \left|\frac{u}{1-\sigma_1}-\sum_{n<u}\left(1-\frac{n}{u}\right)^{-\sigma_1} \right|^2 \frac{du}{u^{2r+1}} \\
&=&\int_1^\infty \left|\frac{u^{1-\sigma_1}}{1-\sigma_1}-\sum_{n<u}\left(u-n\right)^{-\sigma_1} \right|^2 \frac{du}{u^{1+2r-2\sigma_1}} \\
					&=& \sum_{k=1}^\infty \int_k^{k+1} \left|\frac{u^{1-\sigma_1}}{1-\sigma_1}-\sum_{n=1}^k\left(u-n\right)^{-\sigma_1}  \right|^2 \frac{du}{u^{1+2r-2\sigma_1 }}
\end{eqnarray*}
Now, if $u\in (k,k+1]$ and $n<k$, we have
$$
\int_{n-1}^n (u-t)^{-\sigma_1} dt \leq (u-n)^{-\sigma_1} \leq \int_n^{n+1} (u-t)^{-\sigma_1} dt,
$$ 
and so
\begin{equation}\label{equation_zeta}
\frac{1}{u^{\sigma_1}}+\frac{(u-k)^{1-\sigma_1}}{1-\sigma_1}\le \frac{u^{1-\sigma_1}}{1-\sigma_1}-\sum_{n=1}^{k-1}\left(u-n\right)^{-\sigma_1}\le \frac{(u-k+1)^{1-\sigma_1}}{1-\sigma_1}.
\end{equation}
In particular, one has
$$
\frac{u^{1-\sigma_1}}{1-\sigma_1}-\sum_{n=1}^{k}\left(u-n\right)^{-\sigma_1}=-\frac{1}{(u-k)^{\sigma_1}}+O_{\sigma_1}(1)\qquad \big(  u\in (k,k+1] ,\thinspace k\ge 1\big)
$$
so for each $k$, the integral
$$
 \int_k^{k+1} \left|\frac{u^{1-\sigma_1}}{1-\sigma_1}-\sum_{n=1}^k\left(u-n\right)^{-\sigma_1}  \right|^2 du
$$
converges if and only if $2\sigma_1<1$.

Moreover, if $2\sigma_1<1$ these integrals are uniformly bounded with respect to $k$ because using \eqref{equation_zeta}, we have
\begin{align*}
\int_k^{k+1} \left|\frac{u^{1-\sigma_1}}{1-\sigma_1}-\sum_{n=1}^k\left(u-n\right)^{-\sigma_1}  \right|^2 du&\le \int_k^{k+1} \left|\frac{(u-k+1)^{1-\sigma_1}}{1-\sigma_1}+\frac{1}{(u-k)^{\sigma_1}}  \right|^2 du\\
&\le  \int_0^{1} \left|\frac{(u+1)^{1-\sigma_1}}{1-\sigma_1}+\frac{1}{u^{\sigma_1}}  \right|^2 du=\widetilde{C}(\sigma_1)
\end{align*}
From this, we deduce that for $r>\sigma_1$
\begin{align*}
\Vert \psi \Vert^2_2&=\sum_{k=1}^\infty \int_k^{k+1} \left|\frac{u^{1-\sigma_1}}{1-\sigma_1}-\sum_{n=1}^k\left(u-n\right)^{-\sigma_1}  \right|^2 \frac{du}{u^{1+2r-2\sigma_1 }}\\
&\le \sum_{k=1}^\infty \frac{1}{k^{1+2r-2\sigma_1}}\int_k^{k+1} \left|\frac{u^{1-\sigma_1}}{1-\sigma_1}-\sum_{n=1}^k\left(u-n\right)^{-\sigma_1}  \right|^2 du\\
&\le \widetilde{C}(\sigma_1)\sum_{k=1}^\infty \frac{1}{k^{1+2r-2\sigma_1}}<+\infty.
\end{align*}
As far as $\widetilde{C}(\sigma_1)$ is concerned, we have the following bound
\begin{align*}
\widetilde{C}(\sigma_1)& = \int_0^{1} \frac{(u+1)^{2-2\sigma_1}}{(1-\sigma_1)^2}du+\frac{2}{1-\sigma_1}\int_0^{1}\frac{(1+u)^{1-\sigma_1}}{u^{\sigma_1}}du+\int_0^{1}\frac{du}{u^{2\sigma_1}}\\
&\le \int_0^{1} \frac{(u+1)^{2-2\sigma_1}}{(1-\sigma_1)^2}du+ \frac{2^{2-\sigma_1}}{1-\sigma_1}\int_0^{1}\frac{du}{u^{\sigma_1}}+ \int_0^{1}\frac{du}{u^{2\sigma_1}}\\
&\le \frac{2^{3-2\sigma_1}}{(3-2\sigma_1)(1-\sigma_1)^2}+\frac{2^{2-\sigma_1}}{(1-\sigma_1)^2}+\frac{1}{1-2\sigma_1}= C(\sigma_1),
\end{align*}
which proves the first part of the theorem.

On the other hand, suppose that $\sigma_1>0$. (If $\sigma_1 \leq 0$ there is nothing to prove since by hypothesis $r>\sigma_0$ and $\sigma_0=0$ in our case).

It sufficies to find $U(\sigma_1)>0$ such that
$$
\int_k^{k+1} \left|\frac{u^{1-\sigma_1}}{1-\sigma_1}-\sum_{n=1}^k\left(u-n\right)^{-\sigma_1}  \right|^2 du\ge U(\sigma_1)
$$
uniformely in $k$. In that case, we shall have 
\begin{align*}
\Vert \psi \Vert^2_2&=\sum_{k=1}^\infty \int_k^{k+1} \left|\frac{u^{1-\sigma_1}}{1-\sigma_1}-\sum_{n=1}^k\left(u-n\right)^{-\sigma_1}  \right|^2 \frac{du}{u^{1+2r-2\sigma_1 }}\\
&\ge \sum_{k=1}^\infty \frac{1}{(k+1)^{1+2r-2\sigma_1}}\int_k^{k+1} \left|\frac{u^{1-\sigma_1}}{1-\sigma_1}-\sum_{n=1}^k\left(u-n\right)^{-\sigma_1}  \right|^2 du\\
&\ge U(\sigma_1)\sum_{k=1}^\infty \frac{1}{(k+1)^{1+2r-2\sigma_1}}
\end{align*}
hence $r>\sigma_1$. 

It now remains to prove 
$$
\int_k^{k+1} \left|\frac{u^{1-\sigma_1}}{1-\sigma_1}-\sum_{n=1}^k\left(u-n\right)^{-\sigma_1}  \right|^2 du\ge U(\sigma_1)
$$
uniformely in $k$.

From \eqref{equation_zeta}, we see that we have in fact
$$
\frac{u^{1-\sigma_1}}{1-\sigma_1}-\sum_{n=1}^{k}\left(u-n\right)^{-\sigma_1} = -\frac{1}{(u-k)^{\sigma_1}} + \frac{\kappa(u)}{1-\sigma_1}\qquad \big(  u\in (k,k+1] ,\thinspace k\ge 1,\thinspace \sigma_1<1/2\big),
$$
where $0 \leq \kappa(u) \leq (u-k+1)^{1-\sigma_1}$ for all $u \in (k,k+1]$.

We  deduce that for any  $0<\delta<1$
\begin{align*}
\int_k^{k+1} \left|\frac{u^{1-\sigma_1}}{1-\sigma_1}-\sum_{n=1}^k\left(u-n\right)^{-\sigma_1}  \right|^2 du &\geq \int_k^{k+\delta} \left|\frac{u^{1-\sigma_1}}{1-\sigma_1}-\sum_{n=1}^k\left(u-n\right)^{-\sigma_1}  \right|^2 du\\
&\ge \int_k^{k+\delta} \left( \frac{1}{(u-k)^{2\sigma_1}}-2\frac{(u-k+1)^{1-\sigma_1}}{(1-\sigma_1)(u-k)^{\sigma_1}}\right) du  \\
&\ge \int_k^{k+\delta} \left( \frac{1}{(u-k)^{2\sigma_1}}-2\frac{(1+\delta)^{1-\sigma_1}}{(1-\sigma_1)(u-k)^{\sigma_1}}\right) du \\
&\ge \frac{\delta^{1-2\sigma_1}}{1-2\sigma_1}-\frac{2(1+\delta)^{1-\sigma_1}\delta^{1-\sigma_1}}{(1-\sigma_1)^2}\\
&\ge \delta^{1-2\sigma_1}\left( \frac{1}{1-2\sigma_1}-\frac{2(1+\delta)^{1-\sigma_1}\delta^{\sigma_1}}{(1-\sigma_1)^2} \right).\\
\end{align*}
We have the expected result by choosing
$$
U(\sigma_1):=\delta^{1-2\sigma_1}\left( \frac{1}{1-2\sigma_1}-\frac{2(1+\delta)^{1-\sigma_1}\delta^{\sigma_1}}{(1-\sigma_1)^2} \right)
$$
with $\delta>0$ small enough so that
$$
\frac{2(1+\delta)^{1-\sigma_1}\delta^{\sigma_1}}{(1-\sigma_1)^2}<\frac{1}{1-2\sigma_1}.
$$

\end{proof}
\subsubsection{Zero-free discs for $\zeta$}
Let $r>\max(0,\sigma_1)$, $r\not=1$. Then according to Theorem~\ref{pasbesoin}, the function $\psi \in L^2\left((1,\infty), \frac{du}{u^{1+2r}}\right)$. Let  $A=(\alpha,c)$ be a 1-admissible sequence, which means that 
$$
\sum_{j=1}^{\ell(\alpha)} c_j \alpha_j =0,
$$
with $0<\alpha_j\leq 1$, $c_j\in\CC$. Take $\lambda\in\CC$, $\Re(\lambda)>0$. It follows from Theorem~\ref{Thm:disque-sans-zeros} that the Riemann zeta function $s\longmapsto \zeta(s)$ does not vanish in the disc 
\[
r+\left\{\mu\in\CC:\left|\frac{\mu-\lambda}{\mu+\bar\lambda}\right|<\sqrt{1-2\Re(\lambda)d^2_r(\lambda)}\right\},
\] 
where $d_r(\lambda)=\dist(t^{\bar\lambda}\chi_{(0,1)},K_r)$, with $K_r=\Span(f_{A,r}: A \hbox{ a $1$-admissible sequence})$.

If we take only admissible sequences of length $2$, we can easily prove the following.
\begin{Prop}\label{prop:ss-espace} We have
\begin{multline*}
K_r=\Span\left(t^r(c_1\psi\left(\frac{\alpha_1}{t}\right) + c_2\psi\left(\frac{\alpha_2}{t}\right)), \;\; c_1\alpha_1+c_2\alpha_2=0, \; 0<\alpha_j\leq 1\right) = \\ \Span\left(t^r(\psi\left(\frac{\alpha}{t}\right) -\alpha \psi\left(\frac{1}{t}\right)), \;\; 0<\alpha \leq 1\right) .
\end{multline*}
\end{Prop}

In particular, taking $\varphi=\chi_{(0,1)}$ (that is $\sigma_1=0$),  we recover the subspace $\tilde{K}_r$ (see the Introduction) and the result of Nikolski \cite[Theorem 0.1]{Nikolski-AIF}.

We will now try to give more explicit zero free discs for the Riemann zeta function. Applying  Corollary \ref{Cor2:disque-sans-zeros} gives that $s \mapsto \zeta(s)$ does not vanish in the disc
$$
r+\left\{\mu \in \CC \colon \left|\frac{\mu-\lambda}{\mu+\bar{\lambda}}\right| <\sqrt{2\Re(\lambda)} \frac{\left|\zeta(\lambda+r)  \hat{\varphi}(\lambda+r) \sum_{j=1}^{\ell(\alpha)} c_j \alpha_j^{\lambda+r}\right| }{\Vert f_{A,r} \Vert_{L^2_*((0,1), dt/t)}} \right\} .
$$
But the estimate (\ref{remarque_norme}) gives
$$
\Vert f_{A,r} \Vert_{L^2_*((0,1), dt/t)} \leq 
\begin{cases}
\sum_{j=1}^{\ell(\alpha)} |c_j\alpha_j^r| \left( \frac{1}{(1-\sigma_1)\sqrt{2-2r}} + \Vert \psi \Vert_{L^2((1,\infty),\frac{du}{u^{1+2r}})}\right) & \hbox{ if } 0 <r <1, \\
\\
\sum_{j=1}^{\ell(\alpha)} |c_j\alpha_j^r| \left( \frac{(\min_j(\alpha_j)/\max_j(\alpha_j))^{1-r}}{(1-\sigma_1)\sqrt{2r-2}} + \Vert \psi \Vert_{L^2((1,\infty),\frac{du}{u^{1+2r}})} \right) & \hbox{ if } r > 1,
\end{cases}
$$
and thus we deduce (note that the case $r>1$ is less useful):
\begin{Prop} Let $\max(0,\sigma_1)<r<1$ and let $A=(\alpha,c)$ be a 1-admissible sequence. If $\Re(\lambda)>0$, then $s\mapsto \zeta(s)$ does not vanish in the disc
$$
r+\left\{\mu \in \CC \colon \left|\frac{\mu-\lambda}{\mu+\bar{\lambda}}\right| <\sqrt{2\Re(\lambda)} \frac{|\sum_{j=1}^{\ell(\alpha)} c_j \alpha_j^{\lambda+r}|}{\sum_{j=1}^{\ell(\alpha)} |c_j \alpha_j^r|(\Vert \psi \Vert_2 +\frac{1}{(1-\sigma_1)\sqrt{2-2r}})} |\hat{\varphi}(\lambda+r)| |\zeta(\lambda+r)| \right\} .
$$
\end{Prop}
In particular, if we consider only admissible sequences of length $2$, we obtain zero-free discs of the form
\begin{equation}\label{disc_mar}
r+\left\{\mu \in \CC \colon \left|\frac{\mu-\lambda}{\mu+\bar{\lambda}}\right| <\sqrt{2\Re(\lambda)}  \frac{|\alpha^{\lambda+r} -\alpha|}{(\alpha^r+\alpha)(\Vert \psi \Vert_2 + \frac{1}{(1-\sigma_1)\sqrt{2-2r}})} |\hat{\varphi}(\lambda+r)| |\zeta(\lambda+r)| \right\}.
\end{equation}
Since $|\hat{\varphi}(\lambda+r)| = O(1/|\Im(\lambda)|^{1-\sigma_1})$ we obtain larger disc than in (\cite{Nikolski-AIF}) whenever  $\Im(\lambda)$ is large enough and $\sigma_1>0$. We emphasize on the fact that our zero-free discs are explicit.

For example, if we take $\alpha=1/4$ in (\ref{disc_mar}), we obtain the following zero-free region for zeta which can be improved but has the merit to be quite explicit:
\begin{Cor}\label{thm:explicit-zero-RZ} Let $\max(0,\sigma_1)<r<1$ and let $\lambda \in \CC$ such that $\Re(\lambda)>0$. Then the Riemann zeta function does not vanish in the disc
$$
r+ \left\{\mu \in \CC \colon \left|\frac{\mu-\lambda}{\mu+\bar{\lambda}}\right| < F(\lambda,r,\sigma_1) \right\},
$$
where 
\begin{multline*}
F(\lambda,r,\sigma_1)= \\ 
\frac{\sqrt{2\Re(\lambda)} \left|\left(\frac{1}{4}\right)^{\lambda+r}-\frac{1}{4}\right| \big|\Gamma(\lambda+r) \Gamma(1-\sigma_1)\big| }{\left(\left(\frac{1}{4}\right)^r + \frac{1}{4}\right) \left(\sqrt{C(\sigma_1) \zeta(1+2(r-\sigma_1))}+\frac{1}{(1-\sigma_1)\sqrt{2-2r}}\right) \big|\Gamma(\lambda+r+1-\sigma_1)\big|} |\zeta(\lambda+r)|. 
\end{multline*}
\end{Cor}

Recall that the discs in Corollary~\ref{thm:explicit-zero-RZ} are euclidean discs of center $(x,y)$ and radius $R$ where:
\begin{eqnarray*}
x & = & r+ \Re(\lambda) \frac{1+F(\lambda,r,\sigma_1)^2}{1-F(\lambda,r,\sigma_1)^2},  \qquad y  =  \Im(\lambda) \\
R & =& \frac{2\Re(\lambda)F(\lambda,r,\sigma_1)}{1-F(\lambda,r,\sigma_1)^2}.
\end{eqnarray*}
Taking, for example, $\lambda=0.01+50i$, $r=0.49$ and $\sigma_1=0.4$ then a simple evaluation of $F(\lambda,r,\sigma_1)$ implies that $\zeta$ has no zero in the disc of center 
$\frac{1}{2} + 50i$ and radius $3.75 \times 10^{-6}$. Note that $\zeta$ has a zero at $s=\frac{1}{2} +  49.773\dots i$ and at $s=\frac{1}{2} + 52.970\dots i$.

\subsection{Dirichlet $L$-functions} 
Let $\chi$ be a non trivial Dirichlet character of conductor $\cond(\chi)=q$ and $L(\chi,s)=\sum_{n\geq 1} \chi(n)n^{-s}$ be the (degree 1) associated Dirichlet function. This function has an analytic continuation to the whole complex plane. For simplicity, we take $\sigma_1=\sigma_0=0$ and
\[
\varphi(t)=\begin{cases}
1 &\hbox{if }0\leq t <  1,\\
0&\hbox{if }t\geq 1.
\end{cases}
\]
So we have $\psi(u)=-\sum_{n<u} \chi(n)$. Since $\psi$ is bounded, it belongs to $L^2((1,\infty),\frac{du}{u^{1+2r}})$ for all $r>0$. 
The admissibility condition for the sequence $A$ is empty and we may take
$$
f_{A,r}(t)=t^r \sum_{j=1}^{\ell} c_j \sum_{n<\alpha_j/t} \chi(n)
$$
for all $A=(\alpha,c)$ where $\ell\geq 0$, $\alpha \in (0,1]^{\ell}$ and $c\in \CC^{\ell}$. Let $r$ such that $1/2\leq r <1$ and $\lambda=1-r$. For simplicity, 
we write $d_r=d_r(\lambda)$ so that
$$
d_r^2=\min_{c_j, \alpha_j} \left( \int_0^1 \left| t^{1-r} - t^r\sum_{j=1}^{\ell}  c_j \sum_{n<\alpha_j/t} \chi(n)\right|^2 \frac{dt}{t} \right)
$$
Furthermore, we will write $d_{r,\chi}$ if we need to specify the dependance on $\chi$. Remark that we have trivially $d_r^2 \leq 1/(2-2r)$ (just take all the $c_j$ to be 0). In fact, we can be a little bit more precise.
\begin{Prop}\label{Prop:dr-dirichlet}  
We have $d_r^2 < 1/(2-2r)$.
\end{Prop}
\begin{proof} Let $c\in \CC$ and $0<\alpha \leq 1$. A short calculation gives
\begin{multline*}
d_r^2 \leq \int_0^1 \left|t^{1-r} - ct^r \sum_{n<\alpha/t} \chi(n)\right|^2 \frac{dt}{t} = \\
 \frac{1}{2-2r} -2\alpha \Re(cL(\chi,1)) + |c|^2\alpha^{2r+1} \int_1^\infty \left|\sum_{n<u} \chi(n)\right|^2 \frac{du}{u^{\tcr{1}+2r}}.
\end{multline*}
Since we know that $L(\chi,1)\neq 1$ (see for example \cite[p. 37]{kowalski-iwaniec}), one can  choose a suitable $c$ such that the right hand side of the equation above is  less than $1/(2-2r)$.
\end{proof}
Theorem~\ref{Thm:disque-sans-zeros} asserts that $L(\chi,s)$ does not vanish on the disc
\begin{equation}\label{disc}
r+ \left\{\mu \in \CC \colon \left|\frac{\mu+r-1}{\mu-r+1}\right| <\sqrt{1-2(1-r)d_r^2} \right\}.
\end{equation}
Note that the proposition above implies that the disc is not empty. 
We deduce, in particular, that $L(\chi,\sigma)$ does not vanish on the real-interval
$$
\sigma > 1-\frac{\sqrt{1-2(1-r)d_r^2} - ( 1-2(1-r)d_r^2)}{d_r^2}.
$$ 
(We can easily check that the disc (\ref{disc}) contains $\sigma=1$ and it is known that $L(\chi,\sigma)\neq 0$ for $\sigma \geq 1$.) 
Hence, $L(\chi,\sigma)$ does not vanish on
\begin{eqnarray}\label{eq:demi-droite-zero-free}
\sigma > 1- (1-r)\sqrt{1-2(1-r)d_r^2}.
\end{eqnarray}
We expect (by the Riemann hypothesis for Dirichlet functions) that $d_r=0$ which would imply that $L(\chi,s)$ does not vanish on $\Pi_r$. If we had ``only"
$$
d_r^2 \leq \frac{1}{2-2r} - \frac{C^2}{2 (\log q)^2 (1-r)^3}
$$
for some (absolute) constant $C$, it would imply (using \eqref{eq:demi-droite-zero-free}) that $L(\chi,\sigma)$ does not vanish in the real-interval $\sigma >1- C/\log q$. That proves Theorem~\ref{dirichlet} and we also get immediately the following.
\begin{Thm}  If there exists an absolute constant $C$ such that for all (real) character $\chi$ there exists $r \in [1/2,1)$ with 
$$
d_{r,\chi}^2 < \frac{1}{2-2r} - \frac{C^2}{2 (\log \cond(\chi))^2 (1-r)^3},
$$ 
then there is no Siegel's zero for Dirichlet $L$-functions.
\end{Thm}
The previous theorem can be seen as a Beurling-Nyman criterion for Siegel's conjecture. 

\subsubsection{Explicit zero-free discs for $L(\chi,s)$}

We now apply Corollary \ref{Cor2:disque-sans-zeros}. Note that for the special choice of $\varphi$ we have done ($\varphi=\chi_{(0,1)}$), we have $\hat{\varphi}(s)=1/s$. Moreover, since the function $\psi$ is bounded, say $\|\psi\|\leq B$,  we have 
\[
\Vert \psi\Vert_{L^2((1,\infty),du/u^{1+2r})} \leq B/\sqrt{2r}.
\]
Then Corollary~ \ref{Cor2:disque-sans-zeros} asserts that $s\mapsto L(\chi,s)$ does not vanish in the disc
$$
r+\left\{\mu \in \CC \colon \left|\frac{\mu-\lambda}{\mu+\bar{\lambda}}\right| <\sqrt{2\Re(\lambda)} \frac{|L(\chi,\lambda+r)|}{|\lambda+r|} \frac{\sqrt{2r}}{B} \right\}.
$$
Let us remark that the P\'olya-Vinogradov's theorem implies that $B\leq 2\sqrt{q}\log q$ where $q=\cond(\chi)$ is the conductor of $\chi$. Note also that there are some improvements of the P\'olya-Vinogradov's inequality for some characters (see for example \cite{granville-sound}).

\subsection{The Selberg class}

Let $L(s)=\sum_{n\geq 1} a_n n^{-s}$ be a $L$-function in the Selberg class $\mathcal S$ (\cite{anne-TAMS}). We denote by $d$ its degree and by $m_L$  the order of its pole at $s=1$.

As for the Riemann zeta function, we take
\[
\varphi(t)=\begin{cases}
(1-t)^{-\sigma_1} &\hbox{if }0\leq t <  1\\
0&\hbox{if }t\geq 1
\end{cases}
\]
where $\sigma_1 <1/2$. By the Phragmen-Lindel\"of principle, we have that $L(r+it)=O_\varepsilon(t^{\frac{d}{2}(1-r)+\varepsilon})$ for $0<r\leq1/2$. Then we deduce from Theorem \ref{Thm:cns-integrabilite-psi} and \eqref{stirling} that 
$$
\psi \in L^2\left((1,\infty), \frac{du}{u^{1+2r}}\right)
$$
if the following inequality holds
$$
\sigma_1 < \frac{1}{2} - \frac{1-r}{2}d .
$$
In particular, for $r=1/2$, this gives $\sigma_1 < 1/2 - d/4$ which is  better than $\sigma_1<-d/4$ obtained in \cite{anne-TAMS}. Nevertheless, we should mention that for $d<4$, we can in fact take $\sigma_1=0$ (see \cite{Anne-Acta-Arithm}).
\bibliographystyle{alpha}
\bibliography{bibli}
\end{document}